\documentclass{amsart}
\allowdisplaybreaks[4]
\usepackage{graphicx}
\usepackage{color}

\newcommand{\R}{\mathbf{R}}

\newcommand{{\ba}}{\bf a}
\newcommand{\ve}{\varepsilon}
\newcommand{\la}{\lambda}
\newcommand{\La}{\Lambda}
\newcommand{\ga}{\gamma}
\newcommand{\Ga}{\Gamma}
\newcommand{\pa}{\partial}
\newcommand{\ra}{\rightarrow}

\newcommand{\del}{\delta}

\newcommand{\cd}{\cdot}

\newcommand{\al}{\alpha}

\newcommand{\be}{\begin{equation}}
\newcommand{\ee}{\end{equation}}

\newcommand{\nn}{\nonumber}

\newtheorem{lem}{Lemma}{\bf}{\it}
\newtheorem{rem}{Remark}{\it}{\rm}
\newtheorem{ex}{Example}{\it}{\rm}

\newtheorem{proposition}{Proposition}
\newtheorem{corollary}{Corollary}
\numberwithin{theorem}{section}
\numberwithin{lem}{section}
\numberwithin{equation}{section}
\numberwithin{proposition}{section}
\numberwithin{corollary}{section}
\title[LSW coarsening rates]{Bounds on coarsening rates for the Lifschitz-Slyozov-Wagner Equation }

\author{Joseph G. Conlon}

\address{University of Michigan\\ Department of Mathematics\\ Ann Arbor,
  MI 48109-1109}
\email{conlon@umich.edu}


\begin{document}

\maketitle

\begin{abstract}
This paper is concerned with  the large time behavior of solutions to the Lifschitz-Slyozov-Wagner (LSW) system of  equations. Point-wise in time upper and lower bounds on the rate of coarsening are obtained  for solutions with fairly general initial data. These bounds complement the time averaged upper bounds obtained by Dai and Pego, and the point-wise in time upper and lower bounds obtained by Niethammer and Velasquez for solutions with initial data close to a self-similar solution.

\end{abstract}

\section{Introduction.}
In this paper we shall be concerned with the large time behavior of solutions to the  Lifschitz-Slyozov-Wagner (LSW) equations \cite{ls,w}. The LSW equations occur in a variety of contexts \cite{pego, penrose} as a mean field approximation for the evolution of particle clusters of various volumes. Clusters of volume $x>0$ have density $c(x,t)\ge 0$ at time $t>0$. The density evolves according to a linear law, subject to the linear mass conservation constraint as follows:
\begin{eqnarray}
\frac{\pa c(x,t)}{\pa t}&=& \frac{\pa}{\pa x}\left[1-\left(\frac{x}{L(t)}\right)^{1/3}\right]c(x,t),  \quad x>0, \label{A1}\\
\int_0^\infty x c(x,t) dx&=& 1. \label{B1}
\end{eqnarray}
One wishes then to solve (\ref{A1}) for $t>0$ and initial condition $c(x,0) = c_0(x) \ge 0, \ x > 0$, subject to the constraint (\ref{B1}).  The parameter $L(t) > 0$ in (\ref{A1}) is determined by the constraint (\ref{B1}) and is therefore given by the formula,
\be \label{C1}
L(t)^{1/3} = \int^\infty_0 \ x^{1/3} c(x,t)dx \Big/ \int^\infty_0 c(x,t) dx.
\ee
Evidently then $L(t)$ is a measure of the typical cluster volume at time $t$ and the time evolution of the LSW system is in fact non-linear.  Existence and uniqueness of solutions to (\ref{A1}), (\ref{B1}) with given initial data $c_0(x)$ satisfying the constraint has been proven in \cite{cg} for integrable functions $c_0(\cdot)$, and in \cite{np2} for initial data such that $c_0(x)dx$ is an arbitrary Borel probability measure with compact support. In \cite{np3}   the methods of \cite{np2} are further developed to prove existence and uniqueness for initial data such that $c_0(x)dx$ is a Borel probability measure with finite first moment.

The importance of the LSW system lies in the fact that it is one of the simplest systems which is expected to exhibit the phenomenon of {\it coarsening}. Specifically, beginning with rather arbitrary initial data satisfying the constraint (\ref{B1}), one expects the typical cluster volume to increase linearly in time.  This is a consequence of the dilation invariance of the system. That is if the function $c(x,t), \ x,t>0$, is a solution of (\ref{A1}), (\ref{B1}), then for any parameter $\lambda>0$ so also is the function $\lambda^2 c(\lambda x,\lambda t)$. Letting $\Lambda(t)$ be the mean cluster volume at time $t$,
\be \label{I1}
\Lambda(t)=\int_0^\infty xc(x,t)dx \Big/\int_0^\infty c(x,t) dx, \quad t\ge 0,
\ee
one expects  $\Lambda(t)\sim C t$ at large $t$ for some constant $C>0$. The problem of proving that typical cluster volume  increases linearly in time is quite subtle since it is easy to see that the constant $C$ depends on detailed properties of the initial data. In fact if the initial data is a Dirac delta measure then  $C=0$. Less trivially one can construct a  family of self-similar solutions \cite{np1} to (\ref{A1}), (\ref{B1}) depending on a parameter  $\beta$, which may take any value in the interval $0<\beta\le 1$. In that case $\Lambda(t)\sim C(\beta) t$ at large $t$, where $0<C(\beta)<\beta$.

All self-similar solutions to the LSW system have compact support. For a given $\beta, \ 0<\beta\le 1$, the self similar solution has the form,
\be \label{Q1}
c(x,t)=\frac{1}{[1+k(\beta)t]^2} f_\beta\left(\frac{x}{[1+k(\beta)t]}\right) \ .
\ee
Let  the support of $f_\beta(\cdot)$ be the interval $[0,a(\beta)]$. Then as $x\ra a(\beta)$ one has
\begin{eqnarray} \label{Z1}
f_\beta(x) &\sim& K[a(\beta)-x]^{\beta/(1-\beta)-1},  \  \quad  \beta<1,\\
f_\beta(x) &\sim& K\exp\{-1/[a(\beta)-x]\}, \quad  \beta=1. \nn
\end{eqnarray}
It has recently been shown \cite{nv1, nv2} that every self similar solution is a stationary point of an infinite dimensional dynamical system, and that the stationary point is locally asymptotically stable. It follows from this that if the initial data $c_0(\cdot)$ for (\ref{A1}) is sufficiently close to the self similar solution with parameter $\beta$, then $\Lambda(t)\sim C(\beta) t$ at large $t$ where $C(\beta)>0$ is the rate of increase of the average cluster volume of the self-similar solution. The definition of ``closeness" to the self similar solution is quite complicated. A key feature of it is that the initial data $c_0(\cdot)$ should have compact support, and have the same behavior close to the end of its  support as the corresponding self similar solution exhibits in (\ref{Z1}).

Instead of seeking to establish the exact rate of typical cluster volume increase, one can instead simply look for bounds on the typical cluster volume which are linear in time. An upper bound of this nature, which applies to rather general initial data for (\ref{A1}), was  proven by Dai and Pego \cite{dp}. An important ingredient in their proof is an argument of Kohn and Otto \cite{ko}, which has been applied to several systems for which coarsening occurs. The quantity which measures the typical cluster volume is in this case the energy $E(t)$ defined by
\be \label{P1}
 E(t)=\int_0^\infty x^{2/3} c(x,t) dx \ .
 \ee
In view of (\ref{B1}) the ratio $1/E(t)^3$ is a measure of the typical cluster volume. It is shown in \cite{dp} that 
\be \label{J1}
\left[\frac{1}{T}\int_0^T E(t)^2 dt \right]^{-3/2} \le C T, \quad T>1,
\ee 
where $C$ is a constant depending only on the initial data. Thus the Dai-Pego result gives an upper bound on coarsening in a {\it time averaged} sense. In this paper we shall establish bounds on coarsening which are {\it point-wise} in time. In particular we  show in $\S4$ that 
\be \label{K1}
\Lambda(T)\le CT, \quad 1/E(T)^3 \le CT, \quad T>1, 
\ee
where $C$ is a constant depending only on the initial data $c_0(\cdot)$ for (\ref{A1}). Evidently the second inequality of (\ref{K1}) implies (\ref{J1}).
The inequalities (\ref{K1}) apply to a wide range of initial data, even data which is slowly decreasing. An example of this is  $c_0(x)=K_\ve/(1+x)^{2+\ve}, \ \ve>0$,  for constant $K_\ve$ such that (\ref{B1}) holds.  

We also obtain point-wise in time lower bounds on coarsening which are linear in time. These are more difficult to establish than upper bounds because one needs to show that for initial data $c_0(\cdot)$ in some class, the solution $c(\cdot,t)$ of (\ref{A1}), (\ref{B1}) cannot collapse to a Dirac delta distribution for some large time. In $\S4, \ \S5$ we show that
\be \label{O1}
\Lambda(T) \ge CT, \quad T \ge 1,
\ee
for some positive constant $C$ depending only on the initial data $c_0(\cdot)$ for (\ref{A1}). The inequality (\ref{O1}) applies  also for a wide range of initial data,  in particular to the functions $c_0(x)=e^{-x}$ or $c_0(x)=K_\ve/(1+x)^{2+\ve}, \ \ve>0$,  for constant $K_\ve$ such that (\ref{B1}) is satisfied.   The range of initial data for which we are able to prove the lower bound (\ref{O1}) is however slightly smaller than the range of initial data for which we can prove the upper bound (\ref{K1}).  The main difference is that for initial data which decays polynomially at infinity, our proof for the lower bound (Proposition 5.2) requires that there be no oscillation at infinity in the rate of decay. 
At the end of $\S2$ we give an example (Example 1)  of initial data for which we have proved that (\ref{K1}) holds but not (\ref{O1}).

The study of solutions to (\ref{A1}), (\ref{B1}) generally proceeds \cite{np1} by considering the evolution of the function $w(x,t)$,
\be \label{L1}
	w(x,t) = \int^\infty_x \ c(x', t) dx', \ \ x > 0,
 \ee
rather than the evolution of the function $c(x,t)$. The reason for this is that the method of characteristics shows that the time evolution is given by $w(x,t)=w_0(F(x,t))$. The function $F(\cdot,t)$ depends only on the parameter values $L(s),  \ 0\le s\le t$, and $w_0$ is determined from initial data by
\be \label{M1}
w_0(x)=\int_x^\infty c_0(x') dx', \quad x\ge 0. 
\ee
Hence $w_0$ is a continuous nonnegative decreasing function converging to $0$ as $x\ra \infty$, and which may in fact have compact support. The condition (\ref{B1}) at $t=0$ further implies that $w_0$ is also integrable on $(0,\infty)$. 

To obtain an expression for the function $F(\cdot,t)$ one writes the LSW equations (\ref{A1}), (\ref{B1}) using the function (\ref{L1}) as 
\be \label{A3}
\frac {\pa w(x,t)}{\pa t} = \left[ 1 - \left\{ \frac x{L(t)}\right\}^{1/3} \right] \frac{\pa w}{\pa x} (x,t), \ \ x > 0, \ \ t > 0,
\ee

\be \label{B3}
\int^\infty_0 w(x,t)dx = 1.
\ee
The formula (\ref{C1}) for the parameter $L(t)$ now becomes
\be \label{C3}
L(t)^{1/3} = \frac 1 3 \; \int^\infty_0 x^{-2/3} w(x,t) dx \big/ w(0,t).
\ee
Observe from (\ref{C3}) that since $w(\cd,t)$ is nonnegative decreasing the interval $[0, \; L(t)]$ is strictly contained in the support of the function $w(\cd, t)$.

The standard approach to solving (\ref{A3}), (\ref{B3}) is to use the method of characteristics \cite{np1}.  Thus assuming one knows the parameters $L(s), \ 0 \le s \le t$, one solves the ODE
\be \label{D3}
\frac{dx}{ds} = -\left[ 1 - \left\{ \frac x{L(s)} \right\}^{1/3} \right], \ \ 0 < s < t, \ \ x(t) = x.
\ee
Since $w_0(\cdot)$ given by (\ref{M1}) is the initial data for (\ref{A3}),  it follows that $w(x,t) = w_0(x(0))$, whence the function $F(\cdot,t)$ is defined by $F(x,t)=x(0)$.   It is easy to see that $F(x,t)$ is an increasing function of $x$ which satisfies $F(0,t) > 0$.  One can also derive a formula for $\pa F(x,t)/\pa x$ in terms of the solution $x(s), 0 \le s \le t$, of (\ref{D3}).  It is given by the expression
\be \label{E3}
\frac{\pa F(x,t)}{\pa x} = \exp \left[ - \frac 1 3 \ \int^t_0 \ \frac{ds}{x(s)^{2/3} \; L(s)^{1/3}} \right] ,
\ee
from which one concludes that the function $F(\cdot,t)$ is convex. 
The properties of $F(\cdot,t)$ which shall be crucial to our subsequent analysis can be summarized as follows:
\be \label{N1}
F(0,t)>0, \quad 0<\frac{\pa F(x,t)}{\pa x}<1, \quad \frac{\pa^2 F(x,t)}{\pa x^2}>0.
\ee
It follows from (\ref{N1}) that for all $t\ge 0$, $w(x,t)$ is a nonnegative decreasing function of $x$ which converges to 0 as $x \ra \infty$. If $w_0$ has compact support then $w(\cdot,t)$ also has compact support for all $t\ge 0$.  

One way of solving (\ref{A1}), (\ref{B1}) is to take advantage of the dilation invariance of the LSW system.
Thus assuming say $L(0)=1$ in (\ref{C1}), we solve (\ref{A1}), (\ref{B1}) for time $0\le t\le 1$. Then we rescale $c(\cdot,1)$ so as to make $L(1)=1$ and then solve the LSW system again for a unit time 
interval, but now with initial data given by the rescaled $c(\cdot,1)$. Proceeding in this way we can solve the LSW system up to arbitrarily large  time. The advantage of this method  is that as one iterates the process, the solution of the LSW system over a time interval of length one in the rescaled variables should correspond to larger and larger intervals in the original time variable. This is a consequence of the phenomenon of coarsening. 

The method of solving the LSW system described in the previous paragraph is in fact a type of map iteration on integrable nonnegative decreasing functions $w_0(x), \ x\ge 0$, of the form $w_0(x)\ra w_0(F(x)), \ x\ge0$, plus a rescaling to maintain a given normalization. 
 Letting $X_0$ be the positive random variable with cumulative distribution function given by $P(X_0>x)=w_0(x)/w_0(0)$, we see that if $F(\cdot)$ is increasing then the mapping  $w_0(x)\ra w_0(F(x))$ is equivalent to a mapping $X_0\ra T_F(X_0)$ on positive random variables with finite first moment. We will be concerned with obtaining criteria on the function $F(\cdot)$ which implies  the stability of this mapping under arbitrarily large numbers of iterations, for a rather general class of initial variables $X_0$. By stability we mean that if $X_n, \ n=1,2...$ denote the iterations of $X_0$ under $T_F$, the  fluctuation of $X_n$ relative to its mean should neither become arbitrarily large or small.    We shall show uniform boundedness of relative fluctuations by proving that the $X_n, \ n=1,2...$, uniformly satisfy certain {\it reverse} or {\it sharp} Jensen inequalities.

Key to our method is a non-negative function $\beta_0(x), \ x\ge 0$, derived from $w_0(\cdot)$ which we define in $\S2$. This function appears not to have occurred in the literature before. It does however appear implicitly in a paper of Ball et al \cite{bbn}, where they obtain a variational expression for the Fisher information of a probability density function (see $\S2$). It is shown in $\S2$ that if $\sup\beta_0(\cdot)<\infty$, then the corresponding random variable $X_0$ satisfies a reverse Jensen inequality which implies that relative fluctuation of $X_0$ is bounded above. If  on the other hand $\inf\beta_0(\cdot)>0$, then $X_0$ satisfies a sharp Jensen inequality which implies that relative fluctuation of $X_0$ is bounded from below strictly larger than zero. 

In $\S3$ we  show that for certain functions $F(\cdot)$ the mapping $T_F$ is stable, by examining its effect on the $\beta$ function of the random variable. One easily sees that for a convex function $F(\cdot)$ the inequality $\sup T_F(\beta_0)\le \sup\beta_0$ holds. Hence if $F(\cdot)$ is convex, relative fluctuations of $X_n$ cannot become arbitrarily large. We also obtain conditions on $F(\cdot)$ and the initial variable $X_0$ which imply that if $\beta_n$ are the $\beta$ functions corresponding to the variables $X_n$, then $\inf\beta_n$ is uniformly bounded from below strictly larger than zero as $n\ra\infty$. The conditions we impose on the function $F$ and variable $X_0$ in this case are much more restrictive than for the upper bound.

In $\S4$ we turn to the study of  the LSW iteration, obtaining quite general upper bounds on the rate of coarsening, and lower bounds which hold for initial data which is close in some sense to a Dirac delta distribution.   Bounds on the rate of coarsening follow from stability of the LSW iteration, as described in the previous two paragraphs, by virtue of the identity
\be \label{R1}
\frac{d\Lambda(t)}{dt}= \beta(0,t),
\ee
which is proved in Proposition 4.2.
In (\ref{R1}) the quantity $\Lambda(t)$ is the mean cluster volume (\ref{I1})  at time $t$, and $\beta(\cdot,t)$ is the beta function corresponding to the function $w(\cdot,t)$ of (\ref{L1}).
One should note here that for initial data similarly close to a Dirac delta distribution, the results of Niethammer and Velasquez \cite{nv1} also give lower bounds on the rate of coarsening. 
Finally in $\S5$ we show for the LSW system how to obtain lower bounds on the rate of coarsening which are almost as general as our upper bounds.  To do this we first prove that $\beta(\cdot,t)$ is almost monotonic increasing at large $t$. Then on combining this fact with the argument of $\S4$, we obtain a positive lower bound on $\inf\beta(\cdot,t)$, uniform as $t\ra \infty$, which does not require initial data to be close to a Dirac distribution. 

Suppose now that the initial data $w_0(\cdot)$ of  (\ref{M1}) for the LSW model has compact support $0\le x\le a$, and that the $\beta$ function $\beta_0(\cdot)$ corresponding to $w_0(\cdot)$  has a limit at the end of the support satisfying  $0<\lim_{x\ra a}\beta_0(x)<\infty$. Then the results of the paper show that the upper and lower bounds (\ref{K1}), (\ref{O1}) hold in this case. It is of some interest to compare this condition on the initial data with the conditions on initial data required in \cite{np1,nv1,nv2} for a solution to converge to one of the self-similar solutions (\ref{Q1}). In \cite{np1} it is shown (Theorem 5.10) that a {\it necessary} condition for the solution of the LSW model with initial data $w_0(\cdot)$ to converge to the self-similar solution (\ref{Q1}) with $\beta<1$, is that $w_0(x)$ be a {regularly varying function} at $x=a$ with exponent $p=\beta/(1-\beta)$. In $\S2$ we prove that the condition $\lim_{x\ra a}\beta_0(x)=\beta$ implies that $w_0(x)$ is a {regularly varying function} at $x=a$ with exponent $p=\beta/(1-\beta)$.
One can also see from the proof that the condition $\lim_{x\ra a}\beta_0(x)=\beta$ is only slightly stronger than the condition of regular variation at $a$ with exponent  $p=\beta/(1-\beta)$.

An important feature of the methods developed in this paper is their flexibility. In particular they only use the conservation law (\ref{B1}) in a rather general way. The flexibility therefore makes them unsuitable for direct application to the problem of proving {\it asymptotic stability} for self-similar solutions of the LSW system.
There is however a related system for which the methods developed here do yield global asymptotic stability of self-similar solutions. This system is a linearized version of the LSW system, where the power $1/3$ in (\ref{A1}) is replaced by power $1$. The linearized LSW system was first proposed and studied by Carr and Penrose \cite{carr,cp}, who also proved global asymptotic stability of self-similar solutions. 
 A short proof of global asymptotic stability for the linear model using properties of the beta function is given at the end of $\S4$.

\section{The $\beta$ Function}
We shall be interested in the space $\mathbf{E}$ of integrable nonnegative monotonic decreasing functions $w:[0,\infty) \ra [0,\infty)$.
 Evidently the space $\mathbf{E}$ is equivalent to the space of finite Borel measures $\mu$ on $[0,\infty)$ with finite  first moment. We also see that-up to normalization- the space 
 $\mathbf{E}$ is equivalent to the space of random variables $X$ taking values in $(0,\infty)$ with finite first moment $<X> \ <\infty$.  Now from Jensen's inequality, one has that
 \be \label{C2}
 <X^\alpha> \ \le \ <X>^\al, \quad 0\le\al\le1.
 \ee
It is evident that if $0<\al<1$ then strict inequality occurs in (\ref{C2}) except for the trivial random variable $X$  taking a single value with probability $1$.
Thus for $X$ not the trivial random variable and $0<\al<1$, there is an $\eta(\al, X)>0$, depending on $\alpha$ and $X$, such that
\be \label{AC2}
<X^\al> \  \le  \ [1-\eta(\al,X)]<X>^\al.
\ee
We shall refer to (\ref{AC2}) as the {\it sharp} Jensen inequality. 
It is also clear that the variable $X$ satisfies a {\it reverse} Jensen inequality,
  \be \label{D2}
 <X^\alpha> \ \ge \ C(\al,X)<X>^\al, 
 \ee
for  $ 0<\al<1$, and positive constant $C(\al,X)$ depending on $\alpha$ and $X$.

We shall be concerned here with identifying large classes of variables $X$ for which the constants $C(\al,X)$ and  $\eta(\al, X)$ are uniformly bounded from below strictly larger than zero for $X$ in a certain class.
To do this we introduce a function $\beta(\cdot)$ associated with the variable $X$, which has as domain the interval $[0,\|X\|_\infty)$. The interval may be finite i.e. $ \|X\|_\infty<\infty$, or infinite i.e. $\|X\|_\infty=\infty$. Thus let $w\in\mathbf{E}$ correspond to $X$ and $h:[0,\infty)\ra [0,\infty)$ be defined by
\be \label{F2}
h(x)=\int_x^{\infty} w(x') dx', \quad x\ge 0.
\ee
Then $h$ is a non-negative decreasing convex function such that $h(x)\ra 0$ as $x\ra \|X\|_\infty$. If $h$ is $C^2$ on the interval $[0,\|X\|_\infty)$ we may define the $\beta$ function associated to $X$ by
\be \label{G2}
\beta(x) = h''(x) \; h(x) \big/ \; h'(x)^2, \quad 0\le x<\|X\|_\infty.
\ee
Observe that the function $\beta(\cdot)$ in (\ref{G2}) is invariant under multiplication of $h$ by a constant and by dilation scaling.  The transformation (\ref{G2}) can be illustrated graphically (see Figure 1)  as related to Newton's method for solving the equation $h(z)=0$. 
For $x\ge 0$ the function $x\ra x-h(x)/h'(x)$ maps $x$ to Newton's improved value for the solution to $h(z)=0$. The derivative of this function is $\beta(x)$. 

 We have not been able to find an explicit reference to the $\beta$ function (\ref{G2}) in the literature, but it does appear implicitly in \cite {bbn} in a variational expression for the Fisher information of a probability density function. This variational expression plays an important role in the proof \cite{abbn} of the Shannon conjecture for monotonicity of entropy. If the function $h(x), \ -\infty<x<\infty$, is the probability density function for a random variable, then the standard definition for the Fisher information of $h$ is $J(h)$, given  by the formula
$$
J(h)=\int\frac{h'(x)^2}{h(x)} dx.
$$
Evidently if $h'(\cdot)$ converges to $0$ at $\infty$ then $J(h)$ can be alternatively written as
\be \label{WA2}
J(h)=\int\left[\frac{h'(x)^2}{h(x)} -h''(x)\right]dx= \int\frac{h'(x)^2}{h(x)}[1-\beta(x)] dx.
\ee
Suppose now that $h(\cdot)$ is the marginal density  of a joint probability density function,
$$
h(x)=\int w(x,y) dy.
$$
Then $h'(x)^2/h(x)-h''(x)$ has a variational representation in terms of the function $w(x,\cdot)$, and hence by (\ref{WA2}) the Fisher information $J(h)$ has a  variational representation in terms of the function $w(\cdot,\cdot)$.

It is easy to solve (\ref{G2}) for $\beta(x) \equiv \beta =$ constant.  For $0 < \beta < 1$ the function $h$ has compact support.   For $\beta > 1$ it has polynomial decay and for $\beta = 1$ exponential decay.   The solutions normalized so that $h(0)=1, \ h'(0)=-1$ are given by
\begin{eqnarray} \label{J2}
h(x) &=& \left[ 1 - (1 - \beta)x\right]^{1/(1-\beta)}\;, \ \ 0 < \beta < 1, \\
h(x) &=& e^{-x}  \ , \ \ \ \ \ \ \ \ \ \ \ \ \beta = 1,  \nn \\ 
h(x) &=& 1 \Big/ \left[ 1 + (\beta - 1)x\right]^{1/(\beta - 1)}\;, \ \ \beta > 1. \nn
\end{eqnarray}
Observe that if we set $w(x) = -h'(x)$ with $h$ as in (\ref{J2}) then $w$ is an invariant solution of the linearized LSW equations studied in \cite{cp}.  Note also that $\beta=0$ formally corresponds to the trivial random variable, taking a single value with probability $1$.
For general functions $\beta(\cdot)$ it is easy to see that the condition  $\sup\beta(\cdot)\le 1$ is equivalent to the condition that the function $h(\cdot)$ of  (\ref{G2}) is log-concave. If  there is strict inequality $\sup\beta(\cdot)< 1$, then  one can see from the argument of Lemma 2.1 below that $h(\cdot)$ also has compact support. 

The main result of this section shows how the constants in the sharp and reverse Jensen inequalities 
(\ref{AC2}), (\ref{D2}) can be chosen to depend only on upper and lower bounds for the beta function associated with $X$.
\begin{proposition}
Let $X$ be a positive random variable with associated  beta function  given by $\beta(\cdot)$.  Then \\
(a) For $0<\alpha<1, \  0\le \beta_\infty<\infty$, there is a positive constant $C_0(\alpha, \beta_\infty)$, depending only on $\alpha,  \ \beta_\infty$ such that if $\sup\beta(\cdot)\le   \beta_\infty$, the optimal constant $C(\alpha,X)$ in the reverse Jensen inequality (\ref{D2}) satisfies $C(\alpha,X)\ge C_0(\alpha, \beta_\infty)$. \\
(b) For $0<\alpha<1, \  0< \beta_0<\infty$, there is a positive constant $\eta_0(\alpha, \beta_0)$, depending only on $\alpha,  \ \beta_0$ such that if $\inf\beta(\cdot)\ge   \beta_0$, the optimal constant $\eta(\alpha,X)$ in the sharp Jensen inequality (\ref{AC2}) satisfies $\eta(\alpha,X)\ge \eta_0(\alpha, \beta_0)$.
\end{proposition}
We shall prove Proposition 2.1 in a series of lemmas. The proof of (a) is given by the following:
\begin{lem}
Suppose $X$ is a positive random variable and its associated $\beta(\cdot)$ function satisfies
$\sup\beta(\cdot) \le \beta_\infty<\infty$. Then for any $\al, \ 0<\al<1$, the inequality (\ref{D2}) holds for a constant $C(\al,X)=C_1(\beta_\infty)^\al$, where  $C_1(\beta_\infty)$ depends only on $\beta_\infty$.
\end{lem}
\begin{proof}
We begin by obtaining an explicit formula for the function $h$ of (\ref{G2}) in terms of its $\beta(\cdot)$ function. To do this we set $h(x) = \exp[-q(x)]$ in (\ref{G2}), in which case  equation (\ref{G2}) becomes 
\[	-q''(x)\; / \; q'(x)^2= \beta(x) - 1. 	\]
Hence $q'(x)$ is given by the formula
\be \label{M2} 
q'(x) = 1 \; \big/ \; \left[ 1/q'(0) - x + \int^x_0 \beta(x')dx' \right].
\ee
Evidently if the function $\beta(\cdot)$ is associated to the random variable $X$, then \\ $1/q'(0)=<X>$.
Integrating (\ref{M2}) we conclude that
\be \label{N2}
q(x) = q(0) + \int^x_0 \; dz\; \big/ \; \left[ 1/q'(0) - z + \int^z_0 \beta(z')dz' \right], \ 0 \le x < \|X\|_\infty.
\ee
Hence if $w(x)=-h'(x), \ 0\le x<\|X\|_\infty$, then
\be \label{O2}
\frac{w(x)}{w(0)} = \frac{<X>} { \left[<X> - x + \int^x_0 \beta(z')dz' \right]} \ \exp \left[ - \int^x_0 \; \frac{dz} {\left[<X> - z + \int^z_0 \beta(z')dz' \right]} \right].
\ee
From (\ref{O2}) we see that there is a positive constant $C_1(\beta_\infty)$ depending only on $\beta_\infty$, such that
\be \label{P2}
w(x)/w(0)\ge  1/2, \quad 0\le x\le C_1(\beta_\infty) <X>.
\ee
Now (\ref{P2}) implies that
\be \label{Q2}
<X^\al> \ = \ \frac {\al \int_0^\infty x^{-(1-\al)} w(x)dx}{w(0)} \ \ge \  \frac{1}{2} [C_1(\beta_\infty) <X>]^\al.
\ee
\end{proof}
To prove (b) we first obtain a quantitative version of the Jensen inequality (\ref{C2}). 
\begin{lem}
For a positive  random variable $X$ and $\al$ satisfying $0<\al\le 1/2$, there is the inequality,
\be \label{AI2}
E\left[|<X>^\al-X^\al|^{1/\al}\right]  \ \le  \  C(\al) <X>^{1-\al}\left[ <X>^\al-<X^\al>\right],
\ee
and the constant $C(\al)$ satisfies $C(\al)\le 1/\al$. For $0<\al<1$ and $\ve>0$, there is the inequality
\be \label{AJ2}
E\left[|<X>^\al-X^\al|^{1/\al}; \ X>(1+\ve)<X>\right]  \ \le  \  C(\al, \ve) <X>^{1-\al}\left[ <X>^\al-<X^\al>\right],
\ee
where the constant $C(\al, \ve)$ depends only on $\al, \ \ve$.
\end{lem}
\begin{proof}
Consider the function $g(\cdot)$, defined by
$$
g(z)=|z^\al-1|^{1/\al} +z^\al/\al-z, \quad z>0.
$$
If $0<\al\le1/2$ then $g(\cdot)$ has a maximum which is attained at $z=1, \ g(1)=1/\al-1$. Hence
\be \label{AK2}
E\left[ g\left(\frac{X}{<X>}\right)\right] \ \le \ \frac{1}{\al}-1.
\ee
The inequality (\ref{AK2}) implies (\ref{AI2}) with $C(\al)=1/\al$.

To prove (\ref{AJ2}) for $\ve, \del>0$ let $h_{\ve,\del}(\cdot)$ be the function defined by
\begin{eqnarray*}
h_{\ve,\del}(z)&=& z^\al/\al-z, \quad 0<z<1+\ve, \\
h_{\ve,\del}(z)&=& \del g(z)+(1-\del)[z^\al/\al-z], \quad  z>1+\ve.
\end{eqnarray*}
Then for $\del>0$ sufficiently small depending on $\ve,  \al$,  the function $h_{\ve,\del}(\cdot)$has a maximum which is attained at  $z=1, \ h_{\ve,\del}(1)=1/\al-1$. The inequality (\ref{AJ2}) follows now from the inequality,
\be \nn
E\left[ h_{\ve,\del}\left(\frac{X}{<X>}\right)\right] \ \le \ \frac{1}{\al}-1.
\ee
\end{proof}
\begin{rem}
Observe that the inequality (\ref{AI2}) does not hold if $\al>1/2$ for any constant $C(\al)$. To see this let $Z$ be the standard normal variable and $X$ be the variable $1+\sigma Z$ conditioned on $Z>-1/\sigma$. Then for $\sigma$ small the RHS of (\ref{AI2}) behaves like $\sigma^2$ and the LHS like $\sigma^{1/\al}$.
\end{rem}
Statement (b) of Proposition 2.1 is now a consequence of the following:
\begin{lem}
Suppose $X$ is a positive random variable and its associated $\beta(\cdot)$ function satisfies
$\inf \{\beta(\cdot)\} \ge \beta_0>0$. Then for any $\al, \ 0<\al<1$, the inequality (\ref{AC2}) holds for a constant $\eta(\al,X)=\eta_0(\al,\beta_0)$, where  $\eta_0(\al, \beta_0)$ depends only on $\al, \ \beta_0$.
\end{lem}
\begin{proof}
It follows from (\ref{M2}) that 
\be \label{AD2}
E[X|X>x] \ \  \ge  \ \  <X>+\beta_0 x, \quad 0<x<\|X\|_\infty.
\ee
Using the fact that for any $x>0$,   one has $E[X;X>x] \  \ \le \ \ <X>$, we conclude from (\ref{AD2}) that 
\be \label{AE2}
P(X> \ \la <X>) \le 1/(1+\beta_0\la), \quad \la >0.
\ee
Now (\ref{AE2}) implies that for $\xi>0$, one has
\be \label{AF2}
E[X;X \ < \ (1+\xi)<X>] \ \le \ \frac{1+\xi+\beta_0\la^2}{1+\beta_0\la}<X> \ , \quad0< \la<1+\xi.
\ee
Minimizing the RHS of (\ref{AF2}) with respect to $\la>0, \ 0<\la<1+\xi$, we conclude that
\be \label{AG2}
E[X;X \ < \ (1+\xi)<X>] \ \le \   2\frac{1+\xi}{\sqrt{1+(1+\xi)\beta_0}+1}<X> .
\ee 
Note that the RHS of (\ref{AG2}) is strictly less than $<X>$ if $\xi>0$ is sufficiently small, whence the result follows from Lemma 2.2.
\end{proof}
Let $X$ be a positive random variable such that $\|X\|_\infty<\infty$. We shall say that $X$ is {\it regularly varying} with exponent $p\ge 0$ if the function $w(x)=P(X>x)$ satisfies
\be \label{WB2}
\lim_{z\ra 0}  \frac{w(\|X\|_\infty-\la z) }{\la^p w(\|X\|_\infty-z) } =1, \quad {\rm for  \ all \ } \la>0. 
\ee
The following result shows the connection between properties of the beta function for $X$ and the regularly varying property (\ref{WB2}). 
\begin{proposition}
Let $X$ be a positive random variable such that $\|X\|_\infty<\infty$ and $\beta(\cdot)$ be its beta function. If the limit $\lim_{x\ra \|X\|_\infty} \beta(x)=\beta_0$ exists and $\beta_0<1$, then the variable $X$ is regularly varying with exponent $p=\beta_0/(1-\beta_0)$.
\end{proposition}
In view of Lemma 5.7 of \cite{np1}, Proposition 2.2 is a consequence of the following:
\begin{lem} Let $X$ be a positive random variable such that $\|X\|_\infty<\infty$ and $\beta(\cdot)$ be its beta function. Define a function $k(\cdot)$ by $k(\xi)=-\log[P(X>x)]$, where $\xi=-\log[\|X\|_\infty-x], \ x<\|X\|_\infty$. Then for $0\le \beta_0<1$ the limit  $\lim_{x\ra \|X\|_\infty} \beta(x)=\beta_0$ exists, if and only if the limit   $\lim_{\xi\ra \infty} k'(\xi)=p$ of the derivative of $k(\cdot)$ exists with $p=\beta_0/(1-\beta_0)$.
\end{lem}
\begin{proof} Using the notation of Lemma 2.1, we see that $q(\cdot)$ is an increasing function and $\lim_{x\ra \|X\|_\infty} q(x)=\infty$. We also have that 
\be \label{WC2}
\frac{1}{q'(x)}\frac{d}{dx}\left[-\log w(x)\right]= \beta(x).
\ee
Hence on defining a function $g(\cdot)$ by $g(z)=-\log w(x)$ where $z=q(x)$, we see from (\ref{WC2}) that $g'(z)=\beta(x)$. We conclude then from the definition of the function $k(\cdot)$ that 
\be \label {WD2}
k'(\xi)=g'(z)\left[ \|X\|_\infty-x\right]q'(x).
\ee
Suppose now that $\lim_{x\ra \|X\|_\infty} \beta(x)=\beta_0<1$. From (\ref{M2}) we see that 
\be \label{WE2}
q'(x) = 1 \; \big/ \;  \left[\int_x^{\|X\|_\infty} [1- \beta(x') ]dx'\right] ,
\ee
and hence that $\lim_{x\ra \|X\|_\infty} \left[ \|X\|_\infty-x\right]q'(x)=1/(1-\beta_0)$. From (\ref{WD2}) we conclude therefore that $\lim_{\xi\ra \infty} k'(\xi)=\beta_0/(1-\beta_0)$.

Conversely let us suppose that  $\lim_{\xi\ra \infty} k'(\xi)=p$, which is equivalent to the identity,
\be \label{WF2}
\lim_{x\ra \|X\|_\infty} \frac{-w'(x) \left[ \|X\|_\infty-x\right]}{w(x)}=p.
\ee 
We also have that
\be \label{WG2}
 \left[ \|X\|_\infty-x\right]q'(x)=  \frac{w(x) \left[ \|X\|_\infty-x\right]}{h(x)}.
\ee
Next observe that
the function $h(x)$ may be written as
\be \label{WH2}
h(x)=\int_x^{\|X\|_\infty} w(z) \ dz=w(x) \left[ \|X\|_\infty-x\right]+\int_x^{\|X\|_\infty} w'(z) \left[ \|X\|_\infty-z\right] dz.
\ee
It follows then from (\ref{WF2}), (\ref{WH2}) that
\be \label{WI2}
\lim_{x\ra \|X\|_\infty}  \frac{w(x) \left[ \|X\|_\infty-x\right]}{h(x)}=1+p.
\ee
Hence from (\ref{WD2}), (\ref{WG2}) we see that $\lim_{z\ra\infty} g'(z)=p/(1+p)$.
\end{proof}
\begin{proof}[Proof of Proposition 2.2] We use the fact (Lemma 5.7 of \cite{np1}) that $X$ is regularly varying with exponent $p$ if and only if the function $k(\cdot)$ of Lemma 2.4 satisfies
\be \label{WJ2}
\lim_{\xi\ra\infty} \frac{k(\xi+L)-k(\xi)}{L}=p \quad {\rm for \ all \ } L\in\R.
\ee
The result follows from Lemma 2.4. 
\end{proof}
\begin{rem} Observe that the condition (\ref{WJ2}) is only slightly weaker than the condition
 $\lim_{\xi\ra \infty} k'(\xi)=p$. Hence the condition that $X$ is regularly varying, and the condition that the beta function $\beta(x)$  for $X$ has a limit at $x= \|X\|_{\infty}$ which is strictly less than $1$, are almost equivalent.
\end{rem}
We conclude this section with some examples  which illustrate how the beta function for a random variable can oscillate and still remain bounded. The first example is taken from  \cite{cp}.
\begin{ex}
Define the function $h(\cdot)$ by $h(x)=e^{-x}[1+\ve \cos x], \ x\ge 0$. Then it is easy to see from  (\ref{F2}),  (\ref{G2}) that for $|\ve|<1/2$ the function $h(\cdot)$ corresponds to a random variable $X$ with beta function
\be \label{WK2}
\beta(x)=\frac{[1+\ve \cos x][1+2\ve \sin x]}{[1+\ve \cos x+ \ve \sin x]^2}.
\ee
Thus $\beta(\cdot)$ is an oscillatory function with period $2\pi$ which satisfies  $0<\inf\beta(\cdot)<\sup\beta(\cdot)<\infty$.
\end{ex}
\begin{ex}
Define the function $h(\cdot)$ by 
$$
h(x)= (1-x)^{p+1} \left[1+\ve (1-x)^2 \cos \left( \frac{1}{1-x} \right)\right] \quad  0\le x<1. 
$$
Then one can see again from (\ref{F2}),  (\ref{G2}) that for $p>0$ and $|\ve|$ sufficiently small depending only on $p$, that  the function $h(\cdot)$ corresponds to a random variable $X$ with $\|X\|_{\infty}=1$. The beta function for $X$  satisfies
\be \label{WL2}
\beta(x)=\frac{p}{p+1}\left[1-\frac{\ve}{p(p+1)}\cos \left( \frac{1}{1-x} \right) \right] +O[\ve(1-x)].
\ee
Thus $\beta(x)$ does not converge to a limit as $x\ra \|X\|_{\infty}$, but $0<\inf\beta(\cdot)<\sup\beta(\cdot)<1$ for $|\ve|>0$ sufficiently small.
\end{ex}

\section{Iteration of a Map}
In this section we set out the basic methodology which will be followed in the remainder of the paper.
Let $F:[0,\infty)\ra (0,\infty)$ be a $C^2$ function with the properties
\be \label{A2}
0<F'(x)<1,  \ \ F''(x)\ge 0, \quad x\ge 0; \quad \int_0^\infty [1-F'(x)] dx=\infty.
\ee
Observe that the conditions (\ref{A2}) imply that  $F(\cdot)$ has a unique fixed point, so there exists unique  $a>0$ such that $F(a)=a$. The function $F(\cdot)$  induces a mapping $T_F$
on the space $\mathbf{E}$ of integrable nonnegative monotonic decreasing functions $w:[0,\infty) \ra [0,\infty)$ as follows: 
\be \label{B2}
T_F(w)(x)=w(F(x)), \quad x\ge 0. 
\ee
Examples of functions $F$ satisfying (\ref{A2}) are 
\begin{eqnarray} \label{H2}
F(x)& = &(1-\la) + \la x, \ \ 0 < \la < 1,
\\
 \label{I2}
F(x) &= &2^{1/3} + x - (1 + x)^{1/3}.
\end{eqnarray}
Both functions in (\ref{H2}), (\ref{I2}) have the property that $F(1) = 1$.  Linear functions as in (\ref{H2}) occur in the linear version of LSW studied by Carr and Penrose \cite{carr, cp}.  The nonlinear function (\ref{I2}) is more akin to the transformations which occur for the LSW equation.  In particular $F'(x) = 1 - O\left(x^{-2/3} \right)$ for large $x$ as is the case for the LSW transformation $F(x)=F(x,t)$ described in the introduction. 

For a positive random variable $X_0$, we consider variables $X_n, \ n=1,2,..$, which are defined by multiple dilation and iteration of the variable $X_0$ as follows: 
\be \label{W2}
X_{n+1}= T_F(\lambda_nX_n), \quad n=0,1,2,...,
\ee
where the $\la_n>0, \ n=0,1,2...$, are chosen to satisfy $F(0)<\la_n\|X_n\|_\infty, \ n=0,1,2,...$.
We shall obtain a large class of variables $X_0$ for which (\ref{D2}) holds uniformly on the variables $X_n, \ n=0,1,2..$. That is for any $\al, \ 0<\al<1$, there exists a constant $C(\al, X_0)>0$ depending only on $\al$ and $X_0$ such that
\be \label{E2}
 <X_n^\alpha> \ \ge \ C(\al, X_0)<X_n>^\al, \quad n=0,1,2...
 \ee

For a positive random variable $X$  with finite first moment  which satisfies $F(0)<\|X\|_\infty$, the transformation $T_F$ on $X$ induces a corresponding transformation on its $\beta(\cdot)$ function. We denote this transformation also by $T_F$. It is given explicitly by the formula,  
\be \label{K2}
T_F\beta(x)= \beta(F(x)) F'(x) \int^\infty_{F(x)} w(z) \; \frac{dz}{ F'(F^{-1}(z))} \; \bigg/ \int^\infty_{F(x)} \; w(z)dz.
\ee
Since $F$ is convex one concludes from (\ref{K2})  that 
\be \label{L2}
T_F\beta(x) \le \beta(F(x)), \quad  0\le x <\|X\|_\infty, 
\ee
with equality in the case of linear $F$.  Thus when $F$ is linear the constant function is a fixed point of $T_F$, whence the functions (\ref{J2}) are invariant solutions to the linearized LSW equations \cite{carr,cp}.

The usefulness of the $\beta(\cdot)$ function comes from the inequality (\ref{L2}). 
Evidently on combining (\ref{L2}) with Lemma 2.1 we may conclude the following: 
\begin{proposition}
Suppose the function $F:[0,\infty)\ra (0,\infty)$  satisfies the conditions  (\ref{A2}) and $X_0$ is a positive random variable with bounded $\beta(\cdot)$ function. Then if the dilations  $\la_n$  in (\ref{W2}) are chosen to satisfy $F(0)<\la_n\|X_n\|_\infty, \ n=0,1,2..$,  the inequality (\ref{E2}) holds.
\end{proposition}
\begin{proof}
The result follows from Lemma 2.1 and (\ref{L2}).
\end{proof}

Let $X_0$ be a positive random variable and the variables $X_n, \ n=1,2,..$, be defined by (\ref{W2}). We are interested in identifying variables $X_0$ for which
\be \label{R2}
 <X_n^\alpha> \ \le \ [1-\eta(\al,X_0)]<X_n>^\al, \quad n=0,1,2...,
 \ee
 where $0<\al<1$ and $\eta(\al,X_0)$ is a positive constant depending only on $\al$ and $X_0$. The following lemma will enable us to  show that if the $\beta(\cdot)$ function for $X_0$ is bounded away from zero and $\sup\beta(\cdot)$ is sufficiently small then (\ref{R2}) holds. 
 \begin{lem}  Let $F : [0,\infty) \ra (0,\infty)$ satisfy (\ref{A2})  and $\ga_F = \sup\{ x \ge 0 : xF'(x) \le F(x) \}$. Assume $X$ is a positive random variable  with associated $\beta$ function $\beta(\cdot)$,  and that $F(0)<\|X\|_\infty$. Then for any $\ga$, $0 < \ga < \ga_F$, there is a positive continuous function $g_\ga : [0,1] \ra (0,1]$ with $g_\ga(1)=1$ which for any $\nu > 0$ has the property:
\be \label{S2}
\beta(x) \ge \nu g_\ga(x/\|X\|_\infty), \quad 0 \le x < \|X\|_\infty, 
\ee
and $\|X\|_\infty \le F(\ga)$ implies
\[	T_F\beta(x) \ge \nu g_\ga(x/\|T_F(X)\|_\infty), \quad  0 \le x < \|T_F(X)\|_\infty.  \]
\end{lem}
\begin{proof}    We note that since the function $F(x) - xF'(x)$ is decreasing one has $xF'(x) < F(x), \ 0 \le x \le \ga$.  From (\ref{K2}) and the assumption (\ref{S2}) on $\beta$ we have that
\[   T_F\beta(x) \ge \nu g_\ga \left( F(x) / \|X\|_\infty \right) F'(x) \; \big/ F'(\|T_FX\|_\infty),  \ 
\ 0 \le x < \|T_F(X)\|_\infty. \]
Hence it is sufficient to construct the function $g_\ga$ to satisfy
\be \label{T2}
	g_\ga \left( F(x)/\|X\|_\infty\right) F'(x) \; \big/ F'( \|T_F(X)\|_\infty) \ge g_\ga\left( x \; \big/ \|T_FX\|_\infty\right), 
 \ \ 0 \le x < \|T_F(X)\|_\infty.   
\ee
Since we obtain equality in (\ref{T2}) if we set $x=\|T_F(X)\|_\infty$ we need to do a Taylor expansion around this point.  Thus to first order in $\|T_F(X)\|_\infty -x$ the inequality (\ref{T2}) becomes
\begin{multline} \label{U2}
g_\ga(1) F'' \left( \|T_F(X)\|_\infty \right) \le \\
g'_\ga(1) F'\left( \|T_F(X)\|_\infty\right) \big[ \|X\|_\infty - \|T_F(X)\|_\infty F' \left( \|T_F(X)\|_\infty\right) \big] \; \big/ \|X\|_\infty \|T_F(X)\|_\infty.
\end{multline}
By choosing $g'_\ga(1)$ sufficiently large depending only on $\ga$ we obtain strict inequality in (\ref{U2}), whence (\ref{T2}) holds for $x$ close to $\|T_F(X)\|_\infty$.  Observe also that for any $\ve > 0$ there exists $\del > 0$ depending on $\ve$ and $a$ such that 
\be \label{V2}
F(x) \; \big/ \|X\|_\infty \ge x /\|T_F(X)\|_\infty + \del, \quad 0 \le x \le \|T_F(X)\|_\infty - \ve.
\ee
This follows from the fact that $F(x)/x$ is a strictly decreasing function, $0 < x \le \ga$.  In view of (\ref{U2}), (\ref{V2}) one can construct the function $g_\ga$ which satisfies (\ref{T2}). 
\end{proof}
\begin{proposition}
Suppose the function $F:[0,\infty)\ra (0,\infty)$  satisfies the conditions  (\ref{A2}) and $X_0$ is a positive random variable with $\beta(\cdot)$ function satisfying $\inf[\beta(\cdot)]>0$. Assume further that the dilations  $\la_n$ in (\ref{W2}) are chosen to satisfy $F(0)<\la_n\|X_n\|_\infty \le F(\ga) \ n=0,1,2..$, where $\ga$ is as in Lemma 3.1. Then the inequality (\ref{R2}) holds for $0<\al<1$ and some $\eta(\al,X_0)>0$.
\end{proposition}
\begin{proof}
Note that the conditions (\ref{A2}) imply that $F(0)<a_F <F(\ga_F)$, where $a_F$ is the fixed point for $F(\cdot)$. The result follows now from Proposition 2.1  and Lemma 3.1.
\end{proof}
The dilations $\la_n, \ n=0,1,2...$ in (\ref{W2}) can be considered as ``normalizations" of the random variables $X_n, \ n=0,1,2..$. There are various ways of choosing normalizations. The normalization corresponding to the LSW equation is
\be \label{X2}
<(\la_n X_n)^\rho>^{1/\rho}=K(\rho), \quad n=0,1,2...,
\ee
with $\rho=1/3$. We shall show that if $\sup\beta(\cdot)$ is sufficiently small then the $\la_n, \ n=0,1,2,...$, determined by (\ref{X2}) satisfy the conditions of Propositions 3.1 and 3.2.
\begin{lem} Suppose $X_0$ is a positive random variable with finite first moment and bounded $\beta(\cdot)$ function, which also satisfies (\ref{X2}) for $n=0$ and some  $ \la_0, \rho$ satisfying $ \lambda_0>0, \ 0<\rho\le1$. Suppose further that $K(\rho)$ satisfies the inequality  $F(0)<K(\rho) < F(\ga)$, where $\ga$ is as in Lemma 3.1. Then there exists $\beta_\infty>0$ depending on $\rho$ such that if 
$\sup\beta(\cdot)\le \beta_\infty$, there is the inequality  $F(0)<\la_0 \|X_0\|_\infty < F(\ga)$.
\end{lem}
\begin{proof}
Evidently $X=\la_0 X_0$ satisfies $\|X\|_\infty> K(\rho)$, so we just need to show that $\|X\|_\infty< F(\ga)$. If we take $\beta_\infty = 0$ then $\|X\|_\infty=K(\rho)$ in which case the result follows. We thus need to show that by taking $\beta_\infty<1$ sufficiently small we can still achieve 
$\|X\|_\infty < F(\ga)$. To do this let us put
\[	f(x) = x - \int^x_0 \beta(x')dx', \ \ 0 \le x <  \|X\|_\infty,	\]
so $f$ is a monotonic increasing function and $w(x)$ is given by the formula
\be \label{Y2}
w(x) = \exp \left[ - \int^x_0 \; dz /[f(\|X\|_\infty)) - f(z)] \right] \bigg/ [f(\|X\|_\infty) - f(x)].
\ee
If we use now the inequality,
\[	f(\|X\|_\infty) - f(z) \ge f(\|X\|_\infty) - f(x) + (1 - \beta_\infty) (x-z), \quad  0 \le z \le x,	\]
we may conclude from (\ref{Y2}) that $w(x)$ is bounded below as
\be \label{Z2}
w(x) \ge \left[f(\|X\|_\infty) - f(x)\right]^{\beta_\infty /(1-\beta_\infty)} \bigg/  \left[f(\|X\|_\infty) - f(x) +(1-\beta_\infty)x \right]^{1/(1-\beta_\infty)}, \quad 0 \le x < \|X\|_\infty.
\ee
We have now from (\ref{Y2}) that $w(0) = 1/f(\|X\|_\infty)$.  If we then use the inequalities,
\begin{eqnarray*}
f(\|X\|_\infty) - f(x) &\ge& (1 - \beta_\infty) \del \ \|X\|_\infty, \ \ \ 0 \le x \le (1-\del)\|X\|_\infty, \\
f(\|X\|_\infty) - f(x) &+& (1-\beta_\infty)x \le f(\|X\|_\infty), \ \ \ 0 \le x < \|X\|_\infty,
\end{eqnarray*}
we can conclude from (\ref{Z2}) that $w(x)$ satisfies the inequality,
\be \label{AA2}
 w(x) \ge \left[ (1-\beta_\infty)\del \  \|X\|_\infty \; \big/ f(\|X\|_\infty)\right]^{\beta_\infty /(1-\beta_\infty)} w(0), \
 \quad 0\le x\le (1-\del)\|X\|_\infty.
\ee
We evidently also have the inequalities,
\be \label{AB2}
[1 - \beta_\infty]\|X\|_\infty \le f(\|X\|_\infty) \le \|X\|_\infty.
\ee
The result follows from (\ref{AA2}), (\ref{AB2}) since they imply that for any $\ve>0$ there exists $\beta_\infty>0$ such that if $\sup\beta(\cdot)\le \beta_\infty$, then
$$
<X^\rho>=\rho \int_0^\infty x^{\rho-1} w(x) dx/w(0) \ge (1-\ve) \|X\|_\infty^\rho.
$$
\end{proof}
\begin{corollary} 
Suppose the function $F:[0,\infty)\ra (0,\infty)$  satisfies the conditions  (\ref{A2}) and $X_0$ is a positive random variable with $\beta(\cdot)$ function satisfying $0<\inf[\beta(\cdot)]\le \sup[\beta(\cdot)]=\beta_\infty$. Assume the dilation parameters $\la_n, \ n=0,1,2...$, in (\ref{W2})  are  determined by the normalization condition (\ref{X2}) for some $\rho, \ 0<\rho\le 1$, where $F(0)<K(\rho)<F(\ga)$, with $\ga$ as in Lemma  3.1.   Then if $\beta_\infty>0$ is sufficiently small, the inequalities (\ref{E2}) and (\ref{R2}) hold for all $n=0,1,2...$, and $\al$ satisfying $0<\al<1$. 
\end{corollary}
\begin{proof} Follows directly from Propositions 3.1, 3.2 and Lemma 3.2. 
\end{proof}

\section{The LSW Iteration}

 Our first goal in this section will be to give a new proof  for global existence of solutions to (\ref{A3}), (\ref{B3}) by using the beta function (\ref{G2}) introduced in $\S2$. Other more general proofs of global existence for solutions have been given in \cite{cg,np2,np3}.
\begin{lem}  Suppose the initial data  for the LSW system (\ref{A3}), (\ref{B3}) is a non-negative decreasing integrable function $w_0$ satisfying $L(0) = 1$, where $L(t)$ is given by (\ref{C3}).  Further, suppose $\beta_0(\cdot)$ is the beta function derived from $w_0(\cdot)$ by (\ref{F2}), (\ref{G2}) with $w(\cdot)=w_0(\cdot)$, and that $\sup\beta_0(\cdot) \le \beta_\infty<\infty$.  For $\ve,\del > 0$ let $\mathbf{E}_{\ve,\del}$ be the space of continuous functions $L(t), \ 0 \le t \le \del$, with uniform norm which satisfies $L(0)=1, \ (1+\ve)^{-1} \le L(t) \le 1 + \ve, \ 0 \le t \le \del$.  For $L \in \mathbf{E}_{\ve,\del}$ define $T(L)$ by the RHS of (\ref{C3}).  Thus
\be \label{F3}
\left[ T(L)(t)\right]^{1/3} = \frac 1 3 \ \int^\infty_0 x^{-2/3} w(x,t) dx \; \big/ w(0,t), \ \ 0 \le t \le \del,
\ee
where $w(x,t)$ is a solution to (\ref{A3}) with $w(x,0) = w_0(x)$.  Then for $\del$ sufficiently small depending only on $\beta_\infty$, there exists $\ve > 0$ such that $T$ is a contraction on $\mathbf{E}_{\ve,\del}$.
\end{lem}
\begin{proof} We first show that $\ve$ can be chosen so that $T$ maps $\mathbf{E}_{\ve,\del}$ to itself.  Let $X$ be the positive random variable with finite first moment associated with $w_0(\cdot)$ as in $\S2$. Then from Lemma 2.1 and (\ref{C2})  there is a constant $C(\beta_\infty)>0$ depending only on $\beta_\infty$ such that
\be \label{G3}
C(\beta_\infty)<X>^{1/3} \ \le \ <X^{1/3}> \ \le  \ <X>^{1/3}.
\ee
From (\ref{C3}) we see that $<X^{1/3}>^3=L(0) =1$, whence it follows from  (\ref{G3}) that that $<X>  \ =  \ O(1)$.  

Let us assume now that $L(\cdot)\in  \mathbf{E}_{\ve,\del}$ and $\ve$ satisfies $0 < \ve < 1$. Then we see that a solution $x(s), \  0\le s \le \del$, to (\ref{D3}) has the property
\begin{eqnarray} \label{H3}
-1 \le dx/ds &\le& -1/2, \ \ \ \ \ \ 0 \le x(s) \le 1/16, \\
-1\le dx/ds &\le& [2x(s)]^{1/3}, \ \ x(s) \ge 1/16. \nn
\end{eqnarray}
From (\ref{F3}) we have that 
\begin{eqnarray} \label{I3}
[T(L)(t)]^{1/3} &=& \frac 1 3 \int^\infty_0 x^{-2/3} w_0 \big(F(x,t)\big)dx/w_0\big( F(0,t)\big) \\
&=& \frac 1 3 \int^\infty_{F(0,t)} x^{-2/3} \; \frac{w_0(y)}{\pa F(x,t)/\pa x} \ dy/w_0\big( F(0,t) \big), \nn
\end{eqnarray}
where we have made the change of variable $y=F(x,t), \ x \ge 0$, and $\pa F(x,t)/\pa x$ is given by the RHS of (\ref{E3}).  From (\ref{H3}) it follows that $F(0,t) \le t$ whence from (\ref{O2}) we conclude that 
\be \label{J3}
\frac 1{1+C_1\del} \le \frac{w_0(F(0,t))}{w_0(0)} \le 1 + C_1\del, \ \ 0 \le t \le \del,
\ee
provided $0 \le \del \le \del_1$, where $C_1, \del_1$ are positive constants depending only on $\beta_\infty$.  We also see from (\ref{E3}), (\ref{H3}) that there are positive constants $C_2,\del_2$ such that for $0 \le \del \le \del_2$,
\be \label{K3}
\frac 1{1+C_2\del^{1/3}} \le \frac{\pa F(x,t)}{\pa x} < 1 , \quad  0 \le t \le \del.
\ee
This in turn implies that 
\be \label{L3}
\frac 1{1+C_2\del^{1/3}} \le [ F(x,t) - F(0,t)] \big/ x < 1, \quad  x > 0, \ 0 \le t \le \del.
\ee

We conclude from (\ref{I3}) - (\ref{L3}) that there are positive constants $C_3,\del_3$ depending only on $\beta_\infty$ such that if $0 \le \del \le \del_3, \ 0 \le t \le \del$,
\begin{multline} \label{M3}
\frac 1{3[1+C_3\del^{1/3}]}  \int^\infty_{F(0,t)} \left[ y - F(0,t) \right]^{-2/3} \; w_0(y)dy/w_0(0) \le 
\left[ T(L)(t) \right]^{1/3} \ , \\
\left[ T(L)(t) \right]^{1/3} \le \frac {[1+C_3\del^{1/3}]}{3} \; \int^\infty_{F(0,t)} \left[ y - F(0,t) \right]^{-2/3} \; w_0(y)dy/w_0(0) .
\end{multline}
Observe now that for $0 \le t \le \del$, there are the inequalities
\[  \int^\infty_{2\del} \left\{ \left[ y - F(0,t) \right]^{-2/3}  - y^{-2/3}\right\} \; w_0(y)dy/w_0(0) \le C_4\del^{1/3}, \]

\[  \int^{2\del}_0 y^{-2/3} \; w_0(y)dy/w_0(0) + \int^{2\del}_{F(0,t)} \left[ y - F(0,t) \right]^{-2/3} \; w_0(y)dy/w_0(0)
\le C_4\del^{1/3}, \]
for some universal constant $C_4$. We conclude then from (\ref{M3}) that there are positive constants $C_5, \;\del_5$ depending only on $\beta_\infty$ such that for $0 \le \del \le \del_5$,  $0 \le t \le \del$, there is the inequality,
\be \label{N3}
\frac 1 {1 + C_5\del^{1/3}} \le \left[ T(L)(t) \right]^{1/3} \le 1 + C_5 \del^{1/3} .
\ee
Hence $T$ maps $\mathbf{E}_{\ve,\del}$ to itself provided $C_5 \del^{1/3} \le \ve \le 1, \  0 < \del \le \del_5$.

Next we wish to show that $T$ is a contraction on $\mathbf{E}_{\ve,\del}$.  To do this we combine the formulas (\ref{F3}) and (\ref{I3}). Thus we may write
\begin{multline} \label{BA3}
\left[ T(L)(t)\right]^{1/3} =  \frac 1 3 \bigg[ \int^{<X>/2}_0 x^{-2/3} w_0\left(F(x,t)\right) dx \\
+\int^\infty_{F(<X>/2,t)} x^{-2/3} \; \frac{w_0(y)}{\pa F(x,t)/\pa x} \ dy \bigg]
 \; \bigg/ w_0\left(F(0,t)\right) \ .
\end{multline}
Let $L_1,L_2 \in \mathbf{E}_{\ve,\del}$ and $F_1, F_2$ be the corresponding mappings defined by (\ref{D3}).  From (\ref{D3}) we have the inequality,
\begin{multline} \label{O3}
\left| \frac{dx_1}{ds} - \frac{dx_2}{ds}\right| \le |x_1(s) - x_2(s)| \Big/ 3 L_1(s)^{1/3}\min \left[ x_1(s)^{2/3}, x_2(s)^{2/3}\right]  \\
+ x_2(s)^{1/3} \left| L_1(s)^{-1/3} - L_2(s)^{-1/3} \right|,
\end{multline}
where $x_i(s), \ i = 1,2$ are solutions to (\ref{D3}) corresponding to $L_i(s), \  i =  1,2$ respectively.  Letting $\| \cd \|_\del$ be the uniform norm on $\mathbf{E}_{\ve,\del}$ we see from (\ref{O2}) and (\ref{O3}) that there are positive constants $C_1,\del_1$ depending only on $\beta_0$ such that if $0 < \del < \del_1$, then there is the inequality, 
\be \label{P3}
w_0 \left( F_1(x,t) \right) / w_0 \left( F_2(x,t) \right) \le 1 + C_1\del^{1/3} \|L_1 - L_2 \|_\del, \quad  0 \le t \le \del, \ 0 \le x \le  \ <X>/2.
\ee
From (\ref{E3}) there is the inequality, 
\be \label{Q3}
[\pa F_1(x,t)/ \pa x]/[\pa F_2(x,t)/\pa x] \le 1 + C_2\del \|L_1 - L_2 \|_\del, \quad 0 \le t \le \del, \  x \ge  \ 
<X>/2,
\ee
provided $0 < \del \le \del_2$, where $C_2,\del_2$ are constants depending only on $\beta_\infty$.  Observe also that if $y = F(x,t)$ then one has
\be \label{R3}
x \ = \ y \ + \ \int^t_0 [dx/ds]ds, \ \ \ x(0) = y,
\ee
where $x(s), \ 0 \le s \le t$, is a solution of (\ref{D3}).  It follows then from (\ref{O3}), (\ref{R3}) that if $y = F_1(x_1,t) = F_2(x_2, t)$ then there is the inequality,
\be \label{T3}
x_1/x_2 \le 1 + C_3 \del \|L_1 - L_2 \|_\del, \quad  0 \le t \le \del, \ y  \ge  \   <X>/3,
\ee
provided $0 < \del \le \del_3$, where $C_3, \del_3$ are constants depending only on $\beta_0$.  Finally we note that there is the inequality
\be \label{U3}
\left| F_1( <X>/2, t) - F_2( <X>/2, t) \right| \le C_4\del \|L_1 - L_2\|_\del, \quad 0 \le t \le \del,
\ee
provided $0 < \del \le \del_4$, where $C_4, \ \del_4$ are constants depending only on $\beta_\infty$. We may now use (\ref{P3}) to estimate the first term in (\ref{BA3}) and  (\ref{Q3})-(\ref{U3}) to estimate the second term. We conclude  that there are constants $C_5, \ \del_5$ depending only on $\beta_\infty$ such that
\be \label{W3}
[TL_1(t)/TL_2(t)]^{1/3}\le 1+C_5\del^{1/3} \|L_1 - L_2\|_\del. \  \  0 \le t \le \del,
\ee
provided $0 < \del \le \del_5$. Hence if $\del$ is sufficiently small depending only on $\beta_\infty$ the mapping $T$ is a contraction. 
\end{proof}
As a consequence of Lemma 4.1 we obtain global existence for the LSW system (\ref{A3}), (\ref{B3}) and also a bound on the rate of coarsening.
\begin{proposition} Let $\beta_0(\cdot), \; w_0(\cdot)$ be as in Lemma 4.1.  Then there exists a solution $w(x,t), \ x\ge 0, \ t> 0$, of the LSW system (\ref{A3}), (\ref{B3}) with initial data $w_0(\cdot)$.  Further, there is a constant $C(\beta_\infty) > 0$ depending only on $\beta_\infty$ such that $L(t)$ as given in (\ref{C3}) satisfies the inequality $L(t) \le 2L(0) + C(\beta_\infty)t, \ t \ge 0$.
\end{proposition}
\begin{proof}  Observe that by dilation invariance we can assume $L(0) = 1$.  By Lemma 4.1 there exists a solution of (\ref{A3}), (\ref{B3}) for $0 \le t \le \del$.  Consider now the function $\beta(\cdot,\del)$ associated with $w(\cdot,\delta)$. From (\ref{E3}) we see that $F(\cdot,\del)$ is convex, whence by (\ref{L2}) it follows that $\sup\beta(\cdot,\del) \le \beta_\infty$.   We may therefore use Lemma 4.1 to find a solution of (\ref{A3}), (\ref{B3}) for some $t > \del$.  By dilation invariance the interval is $\del \le t \le \del + \del L(1)$.  More generally we can define a sequence of times $t_n, \ n=0,1,2....$, with $t_0 = 0$ and $t_n = t_{n-1} + \del L(t_{n-1})$.  Lemma 4.1 then implies the existence of a solution to (\ref{A3}), (\ref{B3}) in the interval $[0, t_n]$.  In view of the fact that $L(t_n) \le (1+\ve) L(t_{n-1}), \ n \ge 1$, we conclude that $L(t_n) \le (1+\ve)t_n/\del, \ n \ge 1$.  We also have that for $t_n \le s \le t_{n+1}$, $L(s) \le (1+\ve) L(t_n)$, whence it follows that
\be \label{X3}
L(t) \ \le \ (1 + \ve)^2\; t \; / \; \del, \quad t \ge \del.
\ee
To complete the proof of the proposition we need then to show that $\lim_{n \ra \infty} t_n =\infty$.  Let us suppose that $\lim_{n \ra \infty} t_n = t_\infty < \infty$.  Then $L(t), \ 0 \le t < t_\infty$, is a continuous function with the property $\lim_{t\ra t_\infty} L(t) = 0$.  It follows that $L(t), \ 0 \le t < t_\infty$, has a maximum at some point $t_1, \ 0 \le t_1 < t_\infty$.  We can assume wlog that $t_1 = 0$.  Now if the random variable $X$ associated with $w_0(\cdot)$ satisfies $\|X\|_\infty= 1$ then $w_0(x) = 1, 0 \le x \le 1$, whence $L(t) = 1$ for all $t$ which is a contradiction.  Hence we may assume $\|X\|_\infty > 1$, in which case $w_0(1) > 0$.  Since $L(t) \le 1, \ 0 \le t < t_\infty$, we have that $F(1,t) \le 1, \ 0 \le t < t_\infty$.  Since $w(x,t) = w_0(F(x,t))$ it follows from (\ref{C3}) that $\lim\inf_{t \ra t_\infty}L(t) > 0$, again a contradiction.  We conclude that $\lim_{n\ra \infty} t_n = \infty$. 
\end{proof}
We can also see that under the same conditions on the initial data as in Proposition 4.1 that the upper bounds (\ref{K1}) on the rate of coarsening hold. 
\begin{proposition} Let $\beta_0(\cdot), \; w_0(\cdot)$ be as in Lemma 4.1 so $L(0)=1, \  \sup\beta_0(\cdot)\le\beta_\infty$, and  $w(x,t), \ x\ge 0, \ t> 0$,  be the solution of the LSW system (\ref{A3}), (\ref{B3}) with initial data $w_0(\cdot)$.  Then the inequality (\ref{K1}) holds for a constant $C$ of the form $C=C(\beta_\infty)$ depending only on $\beta_\infty$.
\end{proposition}
\begin{proof}
Letting  $\La(t)$ be defined by (\ref{I1}), then it is clear that $\La(t) = 1 / w(0,t)$ where $w(x,t)$ is a solution of the LSW system (\ref{A3}), (\ref{B3}).  Hence we have that
\be \label{Y3}
 \frac {d\La}{dt} = \frac{-1}{w(0,t)^2} \; \frac{\pa w}{\pa t} \; (0,t) = \frac{-1}{w(0,t)^2} \; \frac{\pa w}{\pa x} \; (0,t) = \beta(0,t) ,
\ee
where $\beta(\cdot,t)$ is the function (\ref{G2}) corresponding to $w(\cdot,t)$.  Thus one has $\La(T) \le \La(0) + \beta_\infty T$, whence the first inequality in (\ref{K1}) follows.

To obtain the second inequality in (\ref{K1}) we observe from (\ref{P1}) and Jensen's inequality  
(\ref{C2}) that
\be \nonumber
E(T)=<X^{2/3}> / \La(T) \ \le \ <X>^{2/3} /\La(T)=1/\La(T)^{1/3},
\ee
where we have used the conservation law (\ref{B1}).
The result follows then from the first inequality of (\ref{K1}).
\end{proof}
We may adapt the methodology used in Lemma 3.1 and Lemma 3.2  to obtain a positive lower bound on $\inf\beta(\cdot,t)$ which is uniform in $t$ as $t\ra\infty$, when $\sup\beta_0(\cdot)$ is sufficiently small and  $\liminf\beta_0(\cdot)$ at the end of the support of $w_0(\cdot)$ is positive. This implies by virtue of (\ref{R1})  a lower bound on the rate of coarsening for the LSW system. Rather than applying the methods of $\S3$ directly, we shall here take advantage of the fact that the LSW evolution is continuous in time instead of discrete as in $\S3$. This allows for some simplifications in the proof of the LSW lower bound, but the methodology is subject to the same limitations as that followed in the proof of the lower bounds of $\S3$.

We first make a change of scale so that the average cluster volume $\La(t)$ defined by  (\ref{I1}) is normalized to $1$. Thus let us consider the solution $x(s), 0 \le s \le t$, of (\ref{D3}) with $x(t) = x$.  We put $y(\tau') = x(s)/ \La(s)$, where $\tau', s$  and $\tau,  t$ are related by the change of variable, 
\be \label{C4}  \tau' = \int^s_0 ds' / \La(s'), \ \ \ \ \tau = \int^t_0 ds' / \La(s'),	\ee
and $\Lambda(\cdot)$ is the mean volume function (\ref{I1}).
Then (\ref{D3}) becomes
\be \label{B4} \frac {dy}{d\tau '} = -1 + \ga(\tau ')^{1/3} \; y^{1/3} - c(\tau')y, \quad y(\tau) = y.	\ee
The functions $\ga (\cdot), \ c(\cdot)$ are given by $c(\tau')=\beta(0,s), \  \ga(\tau')=\La(s)/L(s)$, with  $s, \tau'$ being related by (\ref{C4}). Observe from Lemma 2.1 that  $1\le \ga(\cdot)\le C(\beta_\infty)$ for a constant $C(\beta_\infty)$ depending only on $\beta_\infty$. Define now a function $G(y,\tau)$ by $G(y,\tau) =y(0)$ where $y(\tau ')$ is the solution of (\ref{B4}) with $y(\tau) = y$.  Then  if $F$ is the function determining the LSW evolution $w(x,t) = w_0(F(x,t))$,  $F$ and  $G$ are related by the identity 
\be \label{D4}  F(x,t) = \La(0) \ G( x /\La(t), \ \tau ), \ee
where $\tau$ is determined by (\ref{C4}).

\begin{proposition}  Let $\beta_0(\cdot), \ w_0(\cdot)$ with associated random variable $X_0$  be as in Proposition 4.2 and $w(x,t), x \ge 0, t \ge 0$, the solution of the LSW system (\ref{A3}), (\ref{B3}) with initial data $w_0(\cdot)$.  Then there exists $\beta_\infty>0$, such that for $\liminf\{\beta_0(x) :  x\ra\|X_0\|_\infty\}>0$,  the inequality (\ref{O1}) holds for some constant $C > 0$ depending only on $\beta_0(\cdot)$.
\end{proposition}
\begin{proof}  Observe first that  if $\beta(\cdot,t)$ denotes the $\beta$ function for $w(\cdot,t)$ then from (\ref{K2}) one has that 
\be \label{A4}	\beta(x,t) = \beta_0(F(x,t))\frac{\pa F(x,t)}{\pa x} \int^\infty_{F(x,t)}dz\left[ w_0(z) \Big/ \frac{\pa F(x', t)}{\pa x'}\right] \Bigg/ \int^\infty_{F(x,t)} w_0(z) dz \ , \ee
where $ z=F(x',t), \ x'\ge 0$.  Thus to get a lower bound on $\beta(0,t)$ it is sufficient to obtain a lower bound on 
\begin{multline}  \label{ZA4}
\inf\left[\frac{\pa F(0,t)}{\pa x}\Big/ \frac{\pa F(x, t)}{\pa x} \ \  \Big| \  0\le x< \|X_t\|_\infty \right] \\
=\inf \left[\frac{\pa G}{\pa y}(0,\tau)\Big/  \frac{\pa G}{\pa y}(y,\tau)  \ \  \Big| \  0\le y< \|X_t\|_\infty/\La(t)\right] \ ,
\end{multline}
where  $X_t$ is the random variable associated with $w(\cdot,t)$. The interval over which one takes the infimum in the second expression in (\ref{ZA4}) can be taken arbitrarily close to the interval $0\le y \le 1$ by choosing $\beta_\infty>0$ sufficiently small. To see this observe that in  (\ref{M2}) one can identify $1/q'(0)=<X>$. Since the denominator in the expression must  be positive for all $x$ satisfying  $0\le x< \|X\|_\infty$, it follows by letting $x\ra \|X\|_\infty$ that
\be \label{AA3}
<X> \  \le  \  \|X\|_\infty  \ \le \   \ <X> \big/ (1-\beta_\infty).
\ee
Thus the interval $ 0\le y\le \|X_t\|_\infty/\La(t)$ is contained in the interval $0 \le y \le 1/(1-\beta_\infty)$.

Now from (\ref{B4}) we have that $\pa G(y,\tau)/\pa y$ is given by the formula
\be \label{E4}
\frac{\pa G}{\pa y}(y,\tau) = \exp\left[ -\int^\tau_0 \; \frac 1 3 \; \frac{\ga(\tau')^{1/3}} {y(\tau')^{2/3}} d\tau' + \int^\tau_0 c(\tau')d\tau' \right]. \ee
Let us denote by $y_\ve(\tau')$ the solution of (\ref{B4}) with $y_\ve(\tau) = \ve$, whence
\be \label{F4}
\frac{\pa G}{\pa y}(0,\tau) \bigg/ \frac{\pa G}{\pa y}(\ve,\tau) = \exp\left[ - \frac 1 3 \int^\tau_0 \; \ga(\tau')^{1/3} \left\{  \frac 1{y_0(\tau ')^{2/3}} - \frac 1{y_\ve(\tau ')^{2/3} } \right\} d\tau'  \right]. \ee
Hence to obtain a uniform lower bound on the infimum in (\ref{ZA4}) it will be sufficient to show that
\be \label{ZB4}
0<y_\ve(\tau')-y_0(\tau') \le \ve \exp[-k(\beta_\infty)(\tau-\tau')], \quad 0 \le \tau' \le \tau, \ 0 \le \ve \le 1/(1-\beta_\infty),
\ee
for some constant  $k(\beta_\infty)$ depending only on $\beta_\infty$, which is strictly positive for $\beta_\infty>0$ sufficiently small.

To prove (\ref{ZB4}) we consider the differential equation (\ref{B4}). Observe that the expression on the RHS of the equation is given by $-f_{\alpha(\tau')}(\ga(\tau')y)$, where $\alpha(\tau')=c(\tau')/\ga(\tau')$ and $f_\alpha(\cdot)$ is the  function  $f_\alpha(z)=1-z^{1/3}+\alpha z, \  z>0$. It is clear that  $f_\alpha(\cdot)$  is  convex  with at most 2 zeros.  If $\alpha > 4/27=0.1481$ the function is positive.  If $\alpha = 4/27$ it is nonnegative with a single degenerate zero.  If $\alpha< 4/27$ there are 2 nondegenerate zeros. The minimum of $f_\alpha(z)$ occurs for $z= [1/3\alpha]^{3/2}$. Thus if $\alpha< 4/27$ one has that
$f_\alpha( [1/3\alpha]^{3/2})<0$.  In addition one has that
\be \label{ZD4}
f_\alpha(z)<0 \quad {\rm for \ } 1+4\alpha<z< [1/3\alpha]^{3/2}, \quad 0<\alpha< 
0.25\left[(4/3)^{3/5}-1\right]=0.0471.
\ee

We shall assume now that $\beta_\infty<0.0471$ so that the inequality (\ref{ZD4}) holds for the functions $f_{\alpha(\tau')}(\cdot), \ \tau'\ge 0$.  Observe from Lemma 2.1 that there is a constant $C(\beta_\infty)>1$ depending only on $\beta_\infty$ and with the property $\lim_{\beta_\infty\ra 0} C(\beta_\infty)=1$, such that  $\sup\ga(\cdot)\le C(\beta_\infty)$. Since we also have that $\inf \ga(\cdot)\ge 1$, we conclude from (\ref{ZD4}) that
\be  \label{ZF4}
f_{\alpha(\tau')}(\ga(\tau')y)<0, \quad {\rm for \ } \tau' \ge 0, \ 1+4\beta_\infty<y< 1/[C(\beta_\infty)(3\beta_\infty)^{3/2}].
\ee
Evidently we may choose  $\beta_\infty>0$ sufficiently small so that there exists $y=y_\infty>1/(1-\beta_\infty)$  which lies in the interval of (\ref{ZF4}). Thus for any $\ve \le y_\infty$ the solution  $y_\ve(\cdot)$  of (\ref{B4}) satisfies $\sup y_\ve(\cdot) \le y_\infty$.
If we now set  $\phi_\ve(\tau')=y_\ve(\tau')-y_0(\tau')$, we see from (\ref{B4}) that
\be \label{ZC4}
\frac{ d\phi_\ve(\tau')}{d\tau'}  = \\
\phi_\ve(\tau') \left\{ \int^1_0 \ga(\tau' )^{1/3} /3\left[  \la y_0(\tau' ) + (1-\la)y_\ve(\tau' )\right]^{2/3} \; d\la 
-c(\tau') \right\} .
\ee
In view of the bound $\sup y_\ve(\cdot) \le y_\infty$, it follows from (\ref{ZC4}) that
\be \label{ZG4}
\frac{ d\phi_\ve(\tau')}{d\tau'}  \ge  \left[ \frac{1}{3y_\infty^{2/3}}-\beta_\infty\right] \phi_\ve(\tau').
\ee
The inequality (\ref{ZB4}) follows from (\ref{ZG4}).
\end{proof}
We end this section with a short proof of asymptotic stability for the {\it linear} LSW model \cite{carr, cp}.
\begin{proposition}
 Let $\beta_0(\cdot), \ w_0(\cdot)$ with associated random variable $X_0$  be as in Proposition 4.2 and $w(x,t), x \ge 0, t \ge 0$, the solution of the linear LSW system (\ref{A3}), (\ref{B3}) with initial data $w_0(\cdot)$.  Thus the power $1/3$ in (\ref{A3}) is replaced by power $1$.  Suppose in addition that
 $\lim\{\beta_0(x) :  x\ra\|X_0\|_\infty\}=\beta_\infty>0$. Then $\La(t)$ defined by  (\ref{I1}) satisfies $\lim_{T\ra\infty } \La(T)/T=\beta_\infty$.
\end{proposition}
\begin{proof} The result follows from (\ref{R1}) provided we show that  $\lim_{t\ra\infty} \beta(0,t)=\beta_\infty$.  Now for the linear LSW model one has $w(x,t) = w_0(F(x,t))$, where $ F(\cdot,t)$ is a linear function for all $t\ge 0$ and $ \lim_{t\ra\infty} F(0,t)= \|X_0\|_\infty$. Since (\ref{K2}) implies that $ \beta(0,t)=\beta_0(F(0,t))$, the result follows.
\end{proof}

\section{Improved Lower Bounds on the Rate of Coarsening}
Our goal in this section is to remove the smallness restriction on $\sup\beta_0(\cdot)$ which was required in Proposition 4.3 for the proof of the lower bound on the rate of coarsening.
We begin by 
deriving the evolution equation for the function $\beta(x,t)$.  To do this we note that
\[	\beta(x,t) = c(x,t) \int^\infty_x \ w(x', t)dx' \Big/ w(x,t)^2,    \]
where $c(x,t)$ satisfies (\ref{A1}) and $w(x,t)$ equation (\ref{A3}).  It follows that $\beta(x,t)$ is a solution to the equation
\be \label{H4}
\frac {\pa \beta(x,t)}{\pa t} - \left[ 1 - \left\{ \frac x{L(t)} \right\}^{1/3} \right] \frac {\pa \beta}{\pa x}(x,t) = -\beta(x,t) g(x,t), 
\ee
where
\be \label{I4}
g(x,t) = \frac 1{3L(t)^{1/3}} \left\{ \frac 1{x^{2/3}} - \int^{\infty}_x w(x',t) \frac{dx'}{x'^{2/3}} \Big/  \int^{\infty}_x w(x',t)dx \right\}.
\ee
Since $w(\cdot,t)$ is a nonnegative  function $g(x,t)$ is also nonnegative, whence from (\ref{H4}) we may conclude that $\sup\beta(\cdot,t)$ is decreasing.  We can also see from (\ref{H4}), (\ref{I4}) that the time evolution preserves monotonicity of $\beta(\cd,t)$.
\begin{lem}  Suppose $\beta(\cd,t)$ satisfies $\sup\beta(\cd,t) \le 1$.  Then the function $g(\cd,t)$ is monotonic decreasing.
\end{lem}
\begin{proof}  We put
\[	h(x,t) = \int^\infty_x \ w(x', t)dx', \ \ h(x,t) = \exp[-q(x,t)].	\]
From (\ref{M2}) we have that
\be \label{J4} 1\bigg/ [\pa q(x,t)/\pa x] = \La(t) - x + \int^x_0 \ \beta(x',t)dx' .	\ee
It follows from (\ref{J4}) that if $\sup\beta(\cd,t) \le 1$ then $h(\cd,t)$ is an integrable function on $[0,\infty)$  Hence on integration by parts in (\ref{I4}) we have that
\[	g(x,t) = \frac 2{9L(t)^{1/3}} \ \frac 1{h(x,t)} \ \int^\infty_x \ h(x', t) \ \frac{dx'}{x'^{5/3}} \ ,	\]
whence
\be \label{K4}
\frac{\pa g}{\pa x} (x,t) = \frac 2{9L(t)^{1/3}} \left[ -\frac{1}{x^{5/3}} + \frac{w(x,t)}{h(x,t)^2} \int^\infty_x h(x',t) \frac{dx'}{x'^{5/3}} \right].   \ee
Observe from (\ref{J4}) that since $\sup \beta(\cd, t)  \le 1$ we have  $\pa^2 q(x,t) /\pa x^2 \ge 0, \ x > 0$. Hence
\begin{multline*}
\frac{w(x,t)}{h(x,t)^2} \int^\infty_x h(x',t) \frac{dx'}{x'^{5/3}} =
 \frac{\pa q(x,t)}{\pa x} \exp [q(x, t)] \int^\infty_x \exp [-q(x',t)] \frac{dx'}{x'^{5/3}}\\ 
\le \exp [q(x, t)] \frac 1{x^{5/3}}  \int^\infty_x \frac{\pa q}{\pa x'} (x', t) \exp [-q(x',t)]dx' = \frac 1{x^{5/3}}.
\end{multline*}
The result follows from (\ref{K4}). 
\end{proof}
\begin{corollary} Suppose $\beta(\cd,0)$ is monotonic increasing with $\sup\beta(\cd,0) \le 1$.  Then $\beta(\cd, t)$ is monotonic increasing with $\sup\beta(\cd,t) \le 1$ for all $t > 0$.
\end{corollary}
\begin{proof} The solution to (\ref{H4}) is given by the formula,
\be \label{L4}
\beta(x,t) = \beta(F(x,t),0) \exp \left[ - \int^t_0 \ g(x(s),s) ds \right],	\ee
where $x(s), 0 \le s \le t$, is the solution to (\ref{D3}).  Noting that the trajectories of (\ref{D3}) do not intersect in the $(x,t)$ plane it follows from Lemma 5.1 and (\ref{L4}) that $\beta(\cd,t)$ is monotonic increasing. 
\end{proof}
We can rewrite the system (\ref{H4}), (\ref{I4}) in the  variables $(y,\tau)$ used in (\ref{B4}).  
With $y = x/\La(t)$ and  $\tau$ given by (\ref{C4}), we set  $w(x,t) = \La(t)^{-1} w^*(y,\tau)$ so that 
$w^*(\cdot,\tau)$ is normalized as $w^*(0,\tau) = 1$.  The conservation law (\ref{B3}) then becomes 
\be \label{Y4}
\int^\infty_0 w^*(y,\tau) dy = 1.
\ee
We may similarly define the function $\beta^*(\cdot,\cdot)$ by  $\beta(x,t) = \beta^*(y,\tau)$. From (\ref{H4}), (\ref{I4}) it follows that the function $\beta^*(\cdot,\cdot)$ satisfies the equation
\be \label{M4}
\frac {\pa \beta^*}{\pa \tau} (y,\tau) - \left[ 1 - \ga(\tau)^{1/3} \; y^{1/3} + \beta^*(0,\tau)y \right] \frac {\pa \beta^*}{\pa y} (y,\tau) 
= -\beta^*(y,\tau) g^*(y,\tau),
\ee
where $g^*(y,\tau)$ is given by the formula,
\be \label{N4}
g^*(y,\tau) = \frac {\ga({\tau})^{1/3}}{ 3} \left\{ \frac 1{y^{2/3}} - \int^\infty_y w^*(y',\tau) \frac{dy'}{y'^{2/3}}\;\Big/ \int^\infty_y w^*(y', \tau)dy' \right\}.
\ee
If we set
\[  h^*(y,\tau) = \int^\infty_y w^*(y',\tau)dy', \ \ \ h^*(y,\tau) = \exp [-q^*(y,\tau)],	\]
we have that
\be \label{O4}
1/\pa q^*(y,\tau)/\pa y = 1 - y + \int^y_0 \; \beta^*(y',\tau)dy'.   
\ee
Letting $X^*_\tau$ be the random variable associated with the positive decreasing function $w^*(\cdot,\tau)$, then $\beta(\cd,\tau)$ is a function with domain $[0, \|X^*_\tau\|_\infty)$. It has the property that the RHS of (\ref{O4}) is strictly positive for $y <  \|X^*_\tau\|_\infty$ but converges to 0 as $y \ra  \|X^*_\tau\|_\infty$.  We can now use Lemma 5.1 to obtain an improvement of Proposition 4.3.
\begin{proposition}  Let $\beta_0(\cdot), \ w_0(\cdot)$ with associated random variable $X_0$  be as in Proposition 4.3 and $w(x,t),  \ x \ge 0, t \ge 0$, the solution of the LSW system (\ref{A3}), (\ref{B3}) with initial data $w_0(\cdot)$.  Assume there exists $\del>0$  such that $\beta_0(z)\le 1$ for $ \|X\|_\infty-\del\le z < \|X\|_\infty$, and in addition $\liminf\{\beta_0(x) :  x\ra\|X\|_\infty\}$ is positive.  Then the inequality (\ref{O1}) holds for some constant $C > 0$ depending only on $\beta_0(\cdot)$. 
\end{proposition}
\begin{proof} It is evident that there exists $T > 0$ such that   $0<\inf \beta(\cd,T) \le \sup \beta(\cd,T)\le1$.  We can assume therefore wlog that this inequality holds for $T=0$. 
We also assume for the moment that $\beta(\cd,0)$ is monotonic increasing, whence $\beta^*(\cd,\tau)$ is also monotonic increasing, $\tau\ge 0$.  Suppose now that $0 \le \tau_0 < \tau_1$, $\beta^*(0,\tau_0) = \eta >0$ and $\beta^*(0,\tau) \le \eta$ for $\tau_0 \le \tau \le \tau_1$.  We will show there exists $\eta_0 > 0$ such that if $\eta \le \eta_0$ then $\beta^*(0,\tau) \ge \kappa\eta$ for $\tau_0 \le \tau \le \tau_1$, where $\kappa > 0$ is a universal constant.  To see this we consider the value of $\|X^*_\tau\|_\infty$.  Choosing $0 < \eta_0 < 1/6$ to satisfy $1 - 2^{1/3} + 2\eta_0 < 0$ it is clear that if $ \|X^*_\tau\|_\infty\le 2$ then $\|X^*_{\tau'}\|_\infty \le 2$ for $\tau_0 \le \tau' \le \tau$.  Alternatively if $\|X^*_\tau\|_\infty > 2$ it follows from (\ref{O4}) and the monotonicity of $\beta^*(\cd,\tau)$ that $\beta^*(2,\tau) \ge 1/2$.

Just as in (\ref{L4}) we have that for $\tau \ge \tau_0$,
\be \label{P4}
\beta^*(0,\tau) = \beta^*(y_1 (\tau_0), \tau_0) \exp \left[- \int^\tau_{\tau_0} g^*(y_1(\tau '), \tau ') d\tau '\right],
\ee
where $y_1(\tau')$ is the solution to (\ref{B4}) with $y_1(\tau)=0$.  More generally, if $y_1(\tau'), y_2(\tau')$ are two solutions of (\ref{B4}) with $y_1(\tau) = y_1, y_2(\tau) = y_2$, then one sees that for $\tau' < \tau$, 
\begin{multline} \label{Q4}
y_2(\tau') - y_1(\tau') = \\
(y_2 - y_1) \exp \left[ \int^\tau_{\tau '} d\tau'' \left\{ \beta^*(0,\tau'') - \int^1_0 \ga(\tau '')^{1/3} /3\left[  \la y_1(\tau '') + (1-\la)y_2(\tau '')\right]^{2/3} \; d\la \right\} \right].
\end{multline}

It follows from (\ref{Q4}) that if $0 \le y_1 \le y_2 \le 2$, then there exists $\del > 0$ depending only on $\eta_0$ such that
\be \label{R4}
0 \le y_2(\tau ') - y_1(\tau') \le (y_2 - y_1) \exp [-\del(\tau - \tau ')], \ \ \tau_0 \le \tau ' \le \tau.
\ee
Choosing now $y_1 = 0, \ y_2 = \|X^*_\tau\|_\infty \le 2$, in (\ref{R4}) we conclude from (\ref{N4}) that for some constant $A$,
\[	g^*(y_1(\tau'), \tau') \le A \exp [-\del (\tau - \tau')], \ \ \tau' \le \tau -1,   \]
whence we conclude that
\be \label{S4}
\beta^*(0,\tau) \ge \kappa \eta, \ \  \tau_0 \le \tau  \le \tau_1.
\ee
Alternatively suppose $\beta^*(2,\tau) \ge 1/2$.  Then from (\ref{P4}) it follows that
\[	\int^\tau_{\tau_0} g^*(y_2(\tau'), \tau') d\tau' \le \log 2,  \]
where $y_2 = 2$.  Observe from (\ref{K4}), (\ref{N4}) that
\[	|\pa g^*(y, \tau ') / \pa y| \le 2\ga(\tau ')^{1/3} \Big/ 9y^{5/3}, \ \ y > 0.	\]
Since $\| \ga(\cd)\|_\infty < \infty$ we conclude from the last 2 inequalities and (\ref{R4}) that
\[	\int^\tau_{\tau_0} g^*(y_1(\tau'), \tau') d\tau' \le K,	\]
where $y_1 = 0$ and $K$ depends only on $\| \ga(\cd)\|_\infty$.  Hence again (\ref{S4}) holds.

We have proved the result under the assumption that $\beta^*(\cd,0)$ is an increasing function.  To extend it to nonincreasing $\beta^*(\cd,0)$ we observe from (\ref{L4}) that we can choose $T_0$ sufficiently large so that there exists $\del<1$  such that 
\be \label{T4}
\beta^*(y,\tau) \ge (1 - \del) \beta^*(y',\tau), \ \ 0 \le y' < y, \ \ \tau \ge T_0.
\ee
We may now argue as before using (\ref{T4}) to replace the strict monotonicity of $\beta^*(\cd,\tau)$. 
\end{proof}
We consider the evolution of the function $w^*(y,\tau)$ defined after Corollary 5.1.  Since $w(x,t)$ is a solution to (\ref{A3}) it follows that $w^*(y,\tau)$ satisfies the equation,
\be \label{U4}
\frac {\pa w^*}{\pa \tau}\; (y,\tau) - \left[ 1-\ga(\tau)^{1/3} \; y^{1/3} + \beta^*(0,\tau)y\right] \frac {\pa w^*}{\pa y}\; (y,\tau) = \beta^*(0,\tau)w^*(y,\tau).
\ee
Hence $w^*(y,\tau)$ is given in terms of the initial data by
\be \label{V4}
w^*(y,\tau) = w^*\Big( F^*(y,\tau), 0 \Big) \exp \left[ \int^\tau_0 \beta^*(0, \tau')d\tau' \right],
\ee
where $F^*(y,\tau) = y(0)$ and $y(\tau ')$, $0 \le \tau ' \le \tau$, is the solution to (\ref{B4}).  Now the self-similar solutions to the LSW system (\ref{A3}), (\ref{B3}) correspond to solutions of (\ref{U4}) which are independent of $\tau$.  We can easily obtain formulas for these by solving (\ref{U4}).  Thus for $0 < \al < 4/27$ the function $z \ra 1-z^{1/3} + \alpha z, \ z > 0$, has 2 nondegenerate zeros which coalesce to a single degenerate zero as $\al \ra 4/27$.  Let $a_\al > 1$ be the minimum zero and defined the function $\Ga_\al$ by
\[	\Ga_\al(z) = \int^z_0 \ \frac{dz'}{1-z'^{1/3} + \al z'}, \ \ \ \ 0 \le z < a_\al.	\]
For each $\al$ there is a time independent solution $w^*(\cdot)$ of (\ref{U4}),
\be \label{W4}
w^*(y) = \exp \left[ -\al \; \Ga_\al(\ga y)\right], \ \ \ 0 \le y < a_\al/\ga,
\ee
where $\ga$ is given by the formula,
\be \label{X4}
\ga = \int^{a_\al}_0 \exp \left[ -\al \; \Ga_\al(z)\right] dz.
\ee
If we set $\beta^*(0,\tau) \equiv \al\ga, \ \ga(\tau) \equiv \ga$, then the function (\ref{W4}) satisfies (\ref{U4}).  Evidently (\ref{W4}) implies that $w^*(0) = 1$ and (\ref{X4}) that (\ref{Y4}) holds.  Observe that by integrating (\ref{U4}) over the interval $0 < y < a_\al/\ga$ we conclude that
\be \label{Z4}
\int^{a_\al}_0 z^{-2/3} \exp \left[ -\al \; \Ga_\al(z)\right] dz = 3.
\ee

Consider now the function $g_\al(z)$ defined by
\be \label{AA4}
g_\al(z) = \frac \al{1 - z^{1/3} + \al z}  \exp \left[ \al \; \Ga_\al(z)\right] \int^{a_\al}_z \exp \left[ -\al \; \Ga_\al(z')\right] dz', \  \ 0 < z < a_\al.
\ee
Then it is clear that the function $\beta^*(y)$ associated with the function $w^*(y)$ of (\ref{W4}) is given by $\beta^*(y) = g_\al(\ga y)$, $0 < y < a_\al/\ga$.  It follows from the method used in Proposition 5.1 that $\beta^*(y)$ is an increasing function.  We shall prove this separately.

\begin{lem}  Suppose $0 < \al < 4/27$.  Then the function $g_\al(z), \ 0 < z < a_\al$, is monotonic increasing. Furthermore $g_\al(0)=\al\ga$ and  $\lim_{z \ra a_\al} \ g_\al(z) = 3\al a_\al^{2/3} < 1$.
\end{lem}
\begin{proof} It is evident from (\ref{X4}) that $g_\al(0) = \al\ga$.  To find the limit of $g_\al(z)$ as $z \ra a_\al$ we need to expand the function (\ref{AA4}) about $z=a_\al$.  We write
\[	\frac 1{1-z^{1/3} + \al z} = \frac{ 3a^{2/3}_\al} {(a_\al -z)(1-3\al a^{2/3}_\al)} - f_\al(z),   \]
where $f_\al(z)$ is nonnegative and $f_\al(a_\al) = 1/a^{1/3}_\al(1-3\al a^{2/3}_\al)^2$.

Then from (\ref{AA4}) we see that
\be \label{AB4}
g_\al(z) = 3\al a^{2/3}_\al - 2\al(a_\al - z) \Big/ a^{1/3}_\al (2 - 3 \al a^{2/3}_\al) + O[(a_\al - z)^2].
\ee
Note that as $\al \ra 4/27, \ g_\al(a_\al) = 3 \al a^{2/3}_\al \ra 1$ and $g'_\al(a_\al)$ remains bounded.  For $\al \ra 0$, $g'_\al(a_\al)/g_\al(a_\al) \ra 1/3$.  To see that $g_\al$ is monotone increasing we see from (\ref{AA4}) that
\be \label{BZ4}
	\frac d{dz} \left[ \Big( 1 - z^{1/3} + \al z\Big) g_\al(z) \right] = \al g_\al(z) - \al \  ,   
\ee
whence it is sufficient to show that $g_\al(z) \ge 3 \al z^{2/3}, \ 0 < z < a_\al$.  We can see from (\ref{AB4}) that this is the case for $z$ close to $a_\al$.  In fact $g_\al(z) > 3\al z^{2/3}$ is equivalent to 
$3 \al a^{2/3}_\al - 3\al z^{2/3} > g_\al(a_\al) - g_\al(z)$.  By (\ref{AB4}) one has $g'_\al(a_\al) =2\al / a_\al^{1/3}(2 - 3 \al a^{2/3}_\al) < 2\al / a_\al^{1/3}$ whence the inequality holds for $z$ sufficiently close to $a_\al$.  To prove it for all $z, \ 0 < z < a_\al$, observe that
\[  \al/g_\al(z) = \big( 1 - z^{1/3} + \al z\Big) \exp[-\al \Ga(z)] \; \Big/ \; \int^{a_\al}_z \exp[-\al \Ga(z')]dz'.  \]
The inequality follows then if we can show that
\be \label{AC4}
\big( 1 - z^{1/3} + \al z\Big) \exp[-\al \Ga(z)] = \int^{a_\al}_z \frac 1{3z'^{2/3}} \exp[-\al \Ga_\al(z')]dz'.
\ee
Observe that (\ref{AC4}) holds for $z = a_\al$ and also holds for $z=0$ by (\ref{Z4}).  To show it holds for all $z$, $0 < z < a_\al$ we easily verify that the derivatives of both sides of (\ref{AC4}) agree.  Finally observe from (\ref{BZ4}) that since $(1-z^{1/3} + \al z ) g'_\al(z) \big/ g_\al (z) \le 1/3z^{2/3}$ then $\log g_\al(z)$ has total variation which is uniformly bounded for $0 < \al < 4/27$. 
\end{proof}
Next we extend the result of Proposition 5.1 so as to remove the restriction $\sup\beta_0(\cdot) \le 1$, but in removing this restriction we need to impose the extra condition that $\beta_0(\cdot)$ converges to a limit at the end of its support.  
\begin{proposition}
Let $\beta_0(\cdot), \ w_0(\cdot)$ with associated random variable $X_0$  be as in Proposition 4.3 and $w(x,t),  \ x \ge 0, t \ge 0$, the solution of the LSW system (\ref{A3}), (\ref{B3}) with initial data $w_0(\cdot)$.  Assume that $\sup\beta_0(\cdot)<\infty$, and in addition $\lim\{\beta_0(x) :  x\ra\|X\|_\infty\}$ exists and  is positive.  Then the inequality (\ref{O1}) holds for some constant $C > 0$ depending only on $\beta_0(\cdot)$. 
\end{proposition}
\begin{proof}  We consider $w^*(y,\tau)$ which satisfies (\ref{U4}).  Now $w^*(y,0)$ is decreasing and $w^*(0,0) = 1$.  We can therefore define points $y_N(0), N = 0,1,2,....$ with $w^*(y_N(0), 0) = 2^{-N}$.  Let $y_N(\tau'), \tau ' \ge 0$, be the solution of (\ref{B4}) with the specified initial condition $y_N(0)$.  Then from (\ref{V4}) it follows that $w^*(y_N(\tau), \tau) / w^*(y_{N+1}(\tau), \tau) = 2$.  For $N=0,1,2,...$ let $I_N(\tau)$ be the interval $I_N(\tau) = \{ y: y_N(\tau) \le y \le y_{N+1}(\tau)\}$.  If $|I_N(\tau)|$ denotes the length of the interval then one sees that
\be \label{AE4}
| I_N(\tau)| = |I_N(0)| \exp \left\{ -\int^\tau_0 \beta^*(0,\tau')d\tau' + \int^\tau_0 d\tau' \ \ga(\tau')^{1/3} \int^1_0 \frac{d\la}{3[\la y_N(\tau')+ (1-\la)y_{N+1}(\tau')]^{2/3}} \right\}.
\ee
It follows that $|I_N(\tau)| / |I_{N+1}(\tau)|$ is an increasing function of $\tau$.  Suppose now that $w^*(\cdot,0)$ is associated with the random variable $Y$ and that
$\lim\{\beta^*(y,0): y\ra\|Y\|_\infty\} = \beta_1>0$.  Then
\be \label{AD4}
\lim_{N\ra\infty} |I_N(0)| \; \Big/ \; |I_{N+1}(0)| = 2^{1/\beta_1-1}.
\ee

To see that (\ref{AD4}) holds  observe that 
$w^*(y,\tau)$ satisfies the equation,
\be \label{AL4}
\frac {\pa w^*(y,\tau)}{\pa y} \Big/ w^*(y,\tau) = -\beta^*(y,\tau) \pa q^*(y,\tau)/\pa y.
\ee
For $\tau\ge 0, y>0$, let $I_y(\tau)$ be the interval  $I_y(\tau) = \{y' : w^*(y,\tau)/2 < w^*(y',\tau) < w^*(y,\tau)\}$.
From (\ref{AL4}) it follows  that $|I_y(\tau)|$ satisfies the identity,
\be \label{AM4}
\log 2 = \int^{y+|I_y(\tau)|}_y \ \frac{\beta^*(y',\tau)}{A(y,\tau) + \int^{y'}_y \; [\beta^*(y'',\tau) - 1]\; dy''}  \ dy',
\ee
where $A(y,\tau) = 1/\pa q^*(y,\tau)/\pa y>0$.  Let us assume for the moment that $\inf\{\beta^*(y',\tau):y'>y\}>1$. Then from (\ref{AM4}) we see that there exists a $\bar\beta(y,\tau)$ satisfying $\inf\{\beta^*(y',\tau):y'\in I_y(\tau)\}\le \bar\beta(y,\tau)\le\sup\{\beta^*(y',\tau):y'\in I_y(\tau)\}$ such that
\be \label{CS4}
\int^{y+|I_y(\tau)|}_y \; [\beta^*(y',\tau) - 1]\; dy'=A(y,\tau)\left[ 2^{1-1/\bar\beta(y,\tau)}-1\right].
\ee
Setting $z=y+|I_y(\tau)|$  and observing that $A(z,\tau)-A(y,\tau)$ is equal to the LHS of (\ref{CS4}) we conclude that
\be \label{CT4}
\int^{z+|I_z(\tau)|}_{z} \; [\beta^*(z',\tau) - 1]\; dz'=A(y,\tau) 2^{1-1/\bar\beta(y,\tau)}\left[2^{1-1/\bar\beta(z,\tau)}-1\right].
\ee
The identity (\ref{AD4}) for $\beta_1>1$ follows upon taking the ratio of (\ref{CT4}) to (\ref{CS4}) with $\tau=0, y=y_N(0)$, and letting $N\ra\infty$. We similarly see that (\ref{AD4}) holds for $\beta_1<1$. To see that it holds for $\beta_1=1$ one observes from (\ref{AM4}) that $|I_y(\tau)|\sim A(y,\tau)\log 2$ if the function $\beta^*(\cdot,\tau)$ is close to $1$ in the interval $I_y(\tau)$.

In view of (\ref{AD4}) we may  assume that
\be \label{AF4}
 |I_N(\tau)| \; \Big/ \; |I_{N+1}(\tau)| \ge 1/2, \ \ \tau \ge 0,
\ee
provided $y_N(\tau ') > 0, \ 0 \le \tau' \le \tau$.  We define $\beta_N(\tau)$ for all $\tau$ which have the property that $y_N(\tau ') > 0, \ 0 \le \tau' \le \tau$, by the formula
\[	\beta_N(\tau) = \exp \left[ - \int^\tau_0 \ \frac{|I_N(\tau')|}{y_{N+1}(\tau ')^{5/3}} \ d\tau ' \right] .   \]
It is evident from (\ref{P2}) that $y_{N+1}(\tau ) \ge c >0, \ 0 \le \tau' \le \tau$, for some constant $c$.  Hence $\beta_N(\tau)$ is a positive decreasing function of $\tau$.  From (\ref{AF4}) we have that
\be \label{BK4}
	\beta_N(\tau) \le \exp \left[ -\frac 1 2 \; \int^\tau_0 \ \frac{|I_{N+1}(\tau')|}{y_{N+2}(\tau ')^{5/3}} \ d\tau ' \right] = \beta_{N+1}(\tau)^{1/2}.   
\ee
Observe also from (\ref{AE4}) that 
\be \label{AG4}
|I_N(\tau)| \; \big/ \; |I_{N+1}(\tau )| \ge C \big/ \beta_N(\tau)^\al,	
\ee
for constants $C,\al$ satisfying $0 < C,\al < 1$, depending only on $\beta_\infty$.  Evidently $\beta_N(\tau) \ge (C/2)^{1/\al}$ for all $\tau\ge 0$, or there is a $\tau^* > 0$ such that $\beta_N(\tau^*)^\al = C/2$.  In the latter case if $\tau>\tau^*$ we have that
\begin{eqnarray*}
\beta_N(\tau) &\le& \beta_N(\tau^*) \exp \left[ - 2 \; \int^\tau_{\tau^*} \ \frac{|I_{N+1}(\tau')|}{y_{N+2}(\tau ')^{5/3}} \ d\tau ' \right] \\
& =&\beta_N(\tau^*) \beta_{N+1}(\tau)^2 \; \big/ \beta_{N+1}(\tau^*)^2 \le \beta_{N+1}(\tau)^2\big/\beta_N(\tau^*)^3 = (2/C)^{3/\al} \beta_{N+1}(\tau)^2.
\end{eqnarray*}
It follows that if $\beta_{N+1}(\tau) < (C/2)^{3/\al}$ then $\beta_N(\tau) < \beta_{N+1}(\tau)$.  Let us suppose now that for some $\tau\ge 0$ one has $\beta_{N}(\tau) \ge (C/2)^{3/\al}$. We shall show that one has in this case also $\beta_{N+1}(\tau) \ge (C/2)^{3/\al}$. To see this observe that we may assume $\beta_{N}(\tau) < (C/2)^{1/\al}$ since otherwise (\ref{BK4}) implies that $\beta_{N+1}(\tau) > (C/2)^{3/\al}$. Hence there exists $\tau*\le \tau$ with $\beta_N(\tau^*)^\al = C/2$. Since $\beta_{N+1}(\tau) < (C/2)^{3/\al}$ implies that $\beta_N(\tau) < \beta_{N+1}(\tau)$ which is a contradiction, we conclude again that
$\beta_{N+1}(\tau) \ge (C/2)^{3/\al}$.  More generally we see that there exists $N(\tau)\le \infty$ such that $\beta_N(\tau)$ satisfies
\be \label{AH4}
\beta_N(\tau) \ge (C/2)^{3/\al} \ \ \ {\rm if} \ \ \ N \ge N(\tau),	
\ee
$\beta_N(\tau)$ monotonic increasing function of $N$ if $N < N(\tau)$.

We show that the function $g^*(y,\tau)$ of (\ref{N4}) depends locally on $w^*(\cd,\tau)$ near $y$.  Since in (\ref{AF4}) we may replace the RHS by something strictly larger than $1/2$ it follows that 
\be \label{AJ4}
w^*(y,\tau) |I_y(\tau)|/2 < \int^\infty_y w^*(y',\tau)dy' < (1+C) w^*(y,\tau)|I_y(\tau)|
\ee
for some positive constant $C$.  We may bound from below the numerator of the RHS of (\ref{N4}) by
\[  \int^\infty_y \left[ \frac 1{y^{2/3}} - \frac 1{y'^{2/3}} \right] w^*(y',\tau)dy' \ge w^*(y,\tau) |I_y(\tau)|^2 \; \big/ \; 6\left[ y+|I_y(\tau)|\right]^{5/3}.   \] 
From these last two inequalities we obtain a lower bound on $g^*(y,\tau)$,
\be \label{AI4}
g^*(y,\tau) \ge C_1  |I_y(\tau)|\; \big/ \; \left[ y+|I_y(\tau)|\right]^{5/3} ,
\ee
for a positive constant $C_1$.  We may also obtain an upper bound on $g^*(y,\tau)$ of the same form as in (\ref{AI4}) provided we assume the constant $\beta_1$ in (\ref{AD4}) satisfies $\beta_1 < 2$.  For $j=0,1,2....,$ let $y_j \ge y$ be defined by $w(y_j, \tau) = w(y,\tau)/2^j$.  Then one has that
\[   \int^\infty_y \left[ \frac 1{y^{2/3}} - \frac 1{y'^{2/3}} \right] w^*(y',\tau)dy' \le \frac{2w^*(y,\tau)}{3y^{5/3}} \sum^\infty_{k=0} \frac{1}{2^k} \left[ \sum^k_{j=0} |I_{y_j}(\tau)| \right] |I_{y_k}(\tau)|.
\]
Hence if $\beta_1 < 2$ it follows from the fact that $|I_{y_j}(\tau)| \big/ |I_{y_{j+1}}(\tau)|>2^{1/\beta_1-1}>1/\sqrt{2}$ that
\[	\int^\infty_y \left[ \frac 1{y^{2/3}} - \frac 1{y'^{2/3}}\right] w^*(y',\tau)dy' \le Cw^*(y,\tau) |I_y(\tau)|^2 \big/ y^{5/3}.  \]
We therefore have from (\ref{AJ4}) and the previous inequality the following upper bound on $g^*(y,\tau)$,
\be \label{AK4}
g^*(y,\tau) \le C_2 |I_y(\tau)| \big/  y^{5/3} \ \ \ {\rm if} \ \ \ \beta_1 < 2,  
\ee
where $C_2$ is a constant.

Suppose now as in Proposition 5.1 that $0 \le \tau_0 < \tau_1, \ \beta^*(0,\tau_0) = \eta > 0$ and 
$\beta^*(0,\tau) \le \eta$ for $\tau_0 \le \tau \le \tau_1$.  We will show that there exists $\eta_0 > 0$ such that if $\eta \le \eta_0$ then $\beta^*(0,\tau) \ge \kappa \eta^{1+\al}$ for $\tau_0 \le \tau \le \tau_1$, where $\al, \kappa$ are positive constants.   To see this let us define $N_{\min}(\tau)$ as the minimum $N$ such that $y_N(\tau) \ge 0$, whence $\beta_N(\tau)$, $N \ge N_{\min}(\tau)$, are well defined.  Since $h^*(0,\tau) = w^*(0,\tau) = 1$ it follows that $w^*(y,\tau) \le w^*(0,\tau)/y, \ y > 0$, and hence that $y_N(\tau) \le 2$, $y_{N+1}(\tau) \le 4$, for $N = N_{\min}(\tau)$.  In view of (\ref{AH4}), (\ref{AI4}) and the inequality $\beta^*(0,\tau_0) \ge \eta$ we see that $\beta_N(\tau_0) \ge \kappa_1 \eta^{1+\al_1}$ for all $N \ge N_{\min}(\tau_0)$, where $\kappa_1, \al_1$ are positive constants.  Therefore from (\ref{R4}) it follows that $\beta_N(\tau) \ge \kappa_2 \eta^{1+\al_1}$ for some constant $\kappa_2 > 0$ provided $\tau_0 \le \tau \le \tau_1$ and $N = N_{\min}(\tau)$. Hence we may conclude from (\ref{AH4}) that $\beta_N(\tau) \ge \kappa_2 \eta^{1+\al_1}$ for all $N \ge N_{\min}(\tau)$.

If $\beta_1 < 2$ we can see from (\ref{AK4}) that $\beta^*(0,\tau) \ge \kappa\eta^{1+\al}$ for $\tau_0 \le \tau \le \tau_1$.  To see this observe that wlog we may assume  $y_N(\tau) = 0$ when $N = N_{\min}(\tau)$.  From (\ref{R4}) and (\ref{AK4}) it follows that $\beta^*(0,\tau) \ge C\beta^*(y_N(\tau_0), \tau_0)$ for a positive constant $C$ and $N = N_{\min}(\tau)$.  We may assume wlog that 
$y_N(\tau') \ge 1/[2 \inf \ga(\cd)]$ for 
$\tau' \le \tau_0, \ N = N_{\min}(\tau)$.  If also $y_N(\tau') \le 2$ then $|I_N(\tau')| \le 4$, whence $g^*(y_N(\tau'), \tau') \le C_3 |I_N(\tau')| / y_{N+1}(\tau')^{5/3}$ for a constant $C_3$.  Alternatively if $y_N(\tau') \ge 2$ then \\ $w^*(y_N(\tau'), \tau') \le 1/2$, whence it follows that 
$|I_N(\tau')| \le 2y_N(\tau')$.  Thus (\ref{AK4}) implies that 
$g^*(y_N(\tau'), \tau') \le C_4 |I_N(\tau')| \Big/ y_{N+1}(\tau')^{5/3}$ for a constant $C_4$.  We conclude that $\beta^*(y_N(\tau_0), \tau_0) \ge \beta_N(\tau_0)^{\al_2}$ for some $\al_2 > 0$.  The result now follows from the lower bound on $\beta_N(\tau_0)$ already established.

We may also make an argument which does not require the assumption $\beta_1 < 2$.  
First observe from (\ref{AE4}) that if 
$w^*(y_N(\tau),\tau) \le 1/2$ there are constants $C_1, \al_1 > 0$ such that
\be \label{AR4}
|I_N(\tau)| \Big/ |I_{N+1} (\tau)| \le C_1 \Big/ \beta_N(\tau)^{\al_1}.
\ee
Second, we see just in the same way as we obtained the inequality  (\ref{AI4}), that if 
$w^*(y_N(\tau),\tau) \le 1/2$, there is a constant $\al_2 > 0$ such that
\be \label{AP4}
\beta^*(y,\tau) \le \beta^* \left( y_N(\tau), \tau\right) \Big/ \beta_N(\tau)^{\al_2}, \ \ y \in I_N(\tau).
\ee
Suppose now $\tau_0 \le \tau \le \tau_1$ and we assume as before that $y_N(\tau) = 0$ when $N = N_{\min}(\tau)$.  Then it follows from (\ref{AP4}) that
\be \label{AQ4}
\beta^*(y,\tau) \le \beta^*(0,\tau) \Big/ \kappa_3\eta^{1+\al_3}, \ \ y \in I_N(\tau) \cup I_{N+1}(\tau),
\ee
for positive constants $\alpha_3,\kappa_3$, where $N = N_{\min}(\tau) + 1$. Now let us assume that $\sup\{\beta^*(y,\tau) \ | \  y \in I_N(\tau) \cup I_{N+1}(\tau)\}<1$. Then applying (\ref{CS4}), (\ref{CT4}) with $y=y_N(\tau)$ and using (\ref{AQ4}) we see that  $|I_N(\tau)| \Big/ |I_{N+1} (\tau)|$ is bounded below by something larger than the RHS of (\ref{AR4}) unless  $\beta^*(0,\tau) \ge  \kappa_4\eta^{1+\al_4}$ for some positive constants  $\alpha_4,\kappa_4$. The lower bound on $\beta^*(0,\tau)$ of the form $\kappa \eta^{1+\al}$ in the case  $\beta_1 \ge 2$ follows.
\end{proof}

\newpage
\thanks{ {\bf Acknowledgement:} The author would like to thank Peter Smereka and Barbara Niethammer for many helpful conversations. This research was partially supported by NSF
under grants DMS-0500608 and DMS-0553487.

\end{document}